\documentclass{amsart}
\usepackage[mathscr]{eucal}
\usepackage{amsmath,amssymb,hyperref}
\usepackage{graphics}

\usepackage{xcolor}

\usepackage{tikz}
\usetikzlibrary{matrix}

\usepackage[all]{xy}

\newtheorem{thm}{Theorem}[section]
\newtheorem{THM}{Theorem}

\newtheorem{COR}[THM]{Corollary}
\newtheorem{cor}[thm]{Corollary}
\newtheorem{prop}[thm]{Proposition}

\newtheorem{lemma}[thm]{Lemma}

\theoremstyle{definition}

\newtheorem{remark}[thm]{Remark}
\newtheorem{example}[thm]{Example}

\DeclareMathOperator{\Pic}{Pic}

\DeclareMathOperator{\Aut}{Aut}
\DeclareMathOperator{\rank}{rank}
\DeclareMathOperator{\sing}{sing}

\DeclareMathOperator{\id}{id}
\DeclareMathOperator{\PSL}{PSL_2}
\DeclareMathOperator{\SL}{SL_2}
\DeclareMathOperator{\Aff}{Aff}
\DeclareMathOperator{\Id}{Id}
\DeclareMathOperator{\Res}{Res}
\DeclareMathOperator{\Alb}{Alb}
\DeclareMathOperator{\indet}{indet}

\def\C{\mathbb C}
\def\Z{\mathbb Z}
\def\P{\mathbb P}

\def\sl{\mathfrak{sl}_2}

\def\X0{X^{\circ}}

\def\Y0{Y^{\circ}}

\begin{document}

\title[Transversely projective foliations]
{Representations of quasiprojective groups,\\ flat connections and\\ transversely projective foliations}

\author[F.Loray]{Frank LORAY}
\address{IRMAR, Universit\'e de Rennes 1, Campus de Beaulieu, 35042 Rennes Cedex, France}
\email{frank.loray@univ-rennes1.fr}

\author[J.V. Pereira]{Jorge Vit\'{o}rio PEREIRA}
\address{IMPA, Estrada Dona Castorina, 110, Horto, Rio de Janeiro,
Brasil}
\email{jvp@impa.br}

\author[F. Touzet]{Fr\'ed\'eric TOUZET}
\address{IRMAR, Universit\'e de Rennes 1, Campus de Beaulieu, 35042 Rennes Cedex, France}
\email{frederic.touzet@univ-rennes1.fr}


\date{\today}

\thanks{The first author is supported by CNRS, and
the second author, by CNPq and FAPERJ}

\begin{abstract}
The main purpose of this paper is to provide a structure theorem for codimension one singular transversely projective foliations
on projective manifolds. To reach our goal, we firstly extend  Corlette-Simpson's classification of rank two representations
of fundamental groups of quasiprojective manifolds by dropping the hypothesis of quasi-unipotency at infinity.
Secondly  we establish an analogue classification for rank two flat meromorphic connections.
In particular, we prove that a rank two flat meromorphic connection with irregular singularities having non trivial Stokes
projectively factors through a connection over a curve.
\end{abstract}

\maketitle

\setcounter{tocdepth}{1}
\sloppy
\tableofcontents

\section{Introduction}

Let $X$ be a smooth projective manifold over $\C$. A (holomorphic singular) codimension one foliation $\mathcal F$ on $X$
is defined by a non zero rational $1$-form $\omega$ satisfying Frobenius integrability condition $\omega\wedge d\omega=0$.
The foliation is said transversely projective if there are rational $1$-forms $\alpha,\beta$ on $X$ such that
the $\mathfrak{sl}_2$-connection defined on the trivial vector bundle $X\times \C^2$ by
\begin{equation}\label{LaConnexion}
Z\mapsto \nabla Z=dZ+AZ\ \ \ \text{with}\ \ \ A=\begin{pmatrix}\alpha&\beta\\ \omega&-\alpha\end{pmatrix}
\end{equation}
is {\bf flat}: $dA+A\cdot A=0$. Note that this flatness property is intrisically attached to the foliation and does not depend on the choice of the defining rational one form $\omega$. Indeed, if $\mathcal F$ is defined by ${\omega}'=a\omega$ where $a$ is any rational function not identically zero, the corresponding connexion matrix is $$A'=\begin{pmatrix}\alpha+\frac{1}{2}\frac{da}{a}&\frac{1}{a}\beta\\ {\omega}'&-(\alpha +\frac{1}{2}\frac{da}{a})\end{pmatrix}.$$
  This definition, equivalent to \cite{Scardua}, extends to the singular case the classical definition \cite{Godbillon} for smooth foliations.
Outside  the polar divisor of the  connection matrix $A$, the foliation $\mathcal F$ admits distinguished germs
of first integrals taking values in $\P^1$,
well defined up to left composition with elements of $\Aut(\P^1) = \PSL(\C)$. Precisely,
given a local basis of $\nabla$-horizontal sections $B=(b_{ij})\in \mathrm{SL}_2(\mathcal O(U))$
on some open set $U$, i.e. satisfying $dB+A\cdot B=0$, the ratio $\varphi:=b_{21}/b_{22}$ provides 
such a local first integral for $\mathcal F$; changing to another basis $B\cdot B_0$
will have the effect of composing $\varphi$ with a Moebius transformation.

Transversely projective foliations play a singular role in the study of codimension one foliations.
They are precisely those foliations whose Galois groupo\"{\i}d in the sense of Malgrange is small
(see \cite{Casale2, Malgrange2}).
They often occur as exceptions or counter examples \cite{BrunellaMinimal,LinsNeto,FredTransvProj} and played
an important role in our study of foliations with numerically trivial canonical bundle \cite{Croco3}.
For these foliations, one can define a monodromy representation by considering analytic continuation
of distinguished germs of first integrals, making the tranverse pseudo-group into a group (see \cite{Frankpseudo});
it coincides with the (projectivization of the) monodromy of the flat connection $\nabla$.

The goal of this paper is to provide a structure theorem for transversely projective foliations
in the spirit of \cite{CerveauSad,CamachoScardua} and what has been done recently 
 in \cite{GaelJorge} for transversely affine foliations. 

In fact, we mainly work with the connection $\nabla$ defined by  (\ref{LaConnexion}) up to birational bundle transformation.
When it has at worst  regular singularities,  $\nabla$ is characterized by its monodromy representation up to birational bundle transformations \cite{Deligne}.
One of the main ingredients that goes into the proof of our structure theorem  is an extension of Corlette-Simpson's classification of rank two representations of quasiprojective fundamental groups  \cite{Corlette-Simpson} which we
now proceed to explain. 

\subsection{Rank-two representations of quasiprojective fundamental groups}\label{sec:IntroCorletteSimpson}
Let $\X0$ be a  quasiprojective manifold and consider $ X$ a projective compactification of $\X0$
with boundary equal to a  simple normal crossing divisor $D$. If $D_i$ is an irreducible component of $D$ then {by a small loop around $D_i$ we mean a loop} $\gamma : S^1 \to \X0$ that extends to a smooth map $\overline \gamma : \overline {\mathbb D} \to  X$ which intersects $D$ transversely on an unique {smooth point of $D_i$}.
  A representation $\rho : \pi_1(\X0,x) \to \SL(\C)$ is
quasi-unipotent at infinity if  for every irreducible component $D_i$ of $D$ and every small loop $\gamma$ around $D_i$, the conjugacy class of $\rho(\gamma)$ is quasi-unipotent (eigenvalues are roots of the unity).

A representation $\rho : \pi_1(\X0,x) \to \SL(\C)$  {projectively} factors through an orbifold  $Y$ if there exists
a morphism $f:\X0 \to Y$ and a representation $\tilde \rho : \pi_1^{\mathrm{orb}}(Y,f(x)) \to \PSL(\C)$ such that the
diagram
\begin{center}
\begin{tikzpicture}
  \matrix (m) [matrix of math nodes,row sep=2em,column sep=4em,minimum width=2em]
  {
    \pi_1(\X0,x ) & \SL(\C)  \\
      {\pi_1}^{\mathrm{orb}}(Y,f(x) ) & \PSL(\C) \\};
  \path[-stealth]
    (m-1-1) edge node [above] {$\rho$} (m-1-2)
            edge node [left] {$f_*$} (m-2-1)
    (m-1-2) edge node [right] {\text{proj}} (m-2-2)
    (m-2-1)  edge node [below] {$\tilde \rho$} (m-2-2)  ;
\end{tikzpicture}
\end{center}
is commutative (we refer to \cite[Section 2.4]{MM} for the definition and general properties of orbifolds).

A polydisk Shimura modular orbifold is a  quotient $\mathfrak H$ of  a polydisk $\mathbb D^n$ by a  group of the form $U(P,\Phi)$ where
$P$ is a projective module of  rank two over the ring of integers $\mathcal O_L$ of a  totally imaginary  {quadratic extension}
$L$ of a totally real number field $F$; $\Phi$ is a skew hermitian form on $P_L=P \otimes_{\mathcal O_L} L$; and $U(P,\Phi)$  is the subgroup of the $\Phi$-unitary group  $U(P_L,\Phi)$
consisting of elements which preserve $P$. This group acts naturally  on $\mathbb D^n$ where $n$ is half the number of
embeddings $\sigma : L \to \C$ such that the quadratic form $\sqrt{-1} \Phi(v,v)$ is indefinite. The aforementioned action is explained in details  in \cite[\S 9]{Corlette-Simpson}.
 Note that there is one tautological representation
 \[
 \pi_1^{orb}(\mathbb D^n /\mbox{U}(P,\Phi) ) \simeq \mbox{SU}(P,\Phi)/\{\pm \Id\} \hookrightarrow \PSL(\C) \, ,
 \] 
 which induces for each embedding $\sigma: L \to \C$ one tautological representation  $\pi_1^{orb}(\mathbb D^n /\mbox{U}(P,\Phi) ) \to \PSL(\C)$.
The quotients $\mathbb D^n /\mbox{U}(P,\Phi)$ are always  quasiprojective
orbifolds, and when $[L:\mathbb Q] > 2n$ they are projective (i.e. proper/compact) orbifolds.
The archetypical examples satisfying $[L:\mathbb Q] =2n$ are the Hilbert modular orbifolds, which are quasiprojective but
not projective.  We refer again to \cite{Corlette-Simpson} for a thorough discussion and point out that our definition of tautological representations
differs slightly from loc. cit. as they consider polydisk Shimura modular stacks instead of orbifolds and consequently their representations
take values in $\SL(\C)$. Here we are lead to consider representations with values in $\PSL(\C)$ because $\pm \Id \in \mbox{SU}(P,\Phi)$
acts trivially on $\mathbb D^n$.

\begin{thm}[Corlette-Simpson]\label{ThCorletteSimpson}
Suppose that $\X0$ is a quasiprojective manifold and $\rho : \pi_1(\X0, x) \to \SL(\C)$ is
a Zariski dense representation  {which is quasi-unipotent at infinity}.
 Then $\rho$ projectively factors through
 \begin{enumerate}
 \item a morphism $f : \X0 \to Y$ to an orbicurve $Y$ (orbifold of dimension one); or
 \item a morphism $f : \X0 \to \mathfrak H$ to a polydisk Shimura modular orbifold $\mathfrak H$.
 \end{enumerate}
In the latter case, the representation actually projectively factors through one of the tautological representations of $\mathfrak H$.
\end{thm}

Although their hypothesis is natural, as representations  coming from geometry (Gauss-Manin connections) are automatically {quasi-unipotent} at infinity, Corlette and Simpson asked in \cite[Section 12.1]{Corlette-Simpson}
what happens if this assumption is dropped. Our first main result answers this question.

\begin{THM}\label{THM:A}
Suppose that $\X0$ is a quasiprojective manifold and $\rho : \pi_1(\X0, x) \to \SL(\C)$ is
a Zariski dense representation which is not quasi-unipotent at infinity.
Then $\rho$ projectively factors through a morphism $f : \X0 \to Y$ to an
orbicurve  $Y$.
\end{THM}

Our method to deal with representations which are not quasi-unipotent at infinity is considerably more elementary than
the sophisticated arguments needed to deal with the quasi-unipotent case. The non quasi-unipotency allows us to prove the
existence of effective divisors   with topologically
trivial normal bundle at the boundary of $\X0$. We then use Malcev's Theorem combined with a result of Totaro about the existence of
fibrations to produce the factorization.

\medskip

\begin{COR}\label{COR:B}
Suppose that $\X0$ is a quasiprojective manifold and $\rho : \pi_1(\X0, x) \to \SL(\C)$ is
a  representation which is not virtually abelian. Then $\rho$ projectively factors through
\begin{enumerate}
\item a morphism $f : \X0 \to Y$ to an orbicurve $Y$; or
\item a morphism $f : \X0 \to \mathfrak H$ to a polydisk Shimura modular orbifold $\mathfrak H$
equipped with one of its tautological representations.
\end{enumerate}
In case {\rm (2)}, the representation is Zariski dense and  quasi-unipotent at infinity.
\end{COR}
\begin{proof}
The case of Zariski dense representations is covered by Theorem \ref{ThCorletteSimpson} and Theorem \ref{THM:A}.   Non virtually abelian representations in  $\Aff(\C)$ factors through an orbicurve according to \cite{Bartoloetal},  \cite[Theorem 4.1]{GaelJorge} and references therein. 

\end{proof}
\subsection{Riccati foliations}\label{IntroRiccati}
Since we are interested in $\PSL$ rather than $\SL$, we will essentially work with
the projective connection associated to (\ref{LaConnexion}), or more geometrically the foliation $\mathcal H$
induced by $\nabla$-horizontal sections on the total space of the projective bundle $X\times\P^1$.
In an affine chart, 
$\mathcal H$ is defined by the ``Riccati'' pfaffian equation
$$dz+\omega-2\alpha z-\beta z^2=0.$$
More generally, a {\bf Riccati foliation} over a projective manifold $X$ consists of a pair $(\pi:P \to X,\mathcal H)=(P,\mathcal H)$ where
$\pi:P\to X$ is a locally trivial  $\P^1$-fiber bundle  in the Zariski topology (i.e. $P$ is the total space of the projectivization $\mathbb P(E)$ of a rank two vector bundle
$E$) and $\mathcal H$ is a codimension one foliation on $P$  which is transverse to a general fiber of $\pi$.  If the context is clear,
we will omit the $\mathbb P^1$-bundle $P$ from the notation and call $\mathcal H$ a Riccati foliation.

The foliation $\mathcal H$  is defined by the projectivization of horizontal sections of a (non unique) flat meromorphic connection $\nabla$
on $E$: it is the phase portrait of the projective connection $\mathbb P(E,\nabla)$.
The connection $\nabla$ is uniquely determined by $\mathcal H$ and by its trace on $\det(E)$.
We say that the Riccati foliation $\mathcal H$ has {\bf regular singularities} if it can be lifted to a meromorphic connection $\nabla$
with regular singularities (see \cite[chap.II, Definition 4.2]{Deligne}),
and {\bf irregular} if not.

We will say that a Riccati foliation $( P, \mathcal H)$  over $X$ factors through a projective manifold $X'$
if  there exists  a Riccati foliation $( \pi':P'\to X', \mathcal H')$ over $X'$,
and  rational maps  $\phi : X \dashrightarrow X'$ , $\Phi: P \dashrightarrow P'$   such that $\pi' \circ \Phi = \phi \circ \pi$, $\Phi$ has degree one when restricted to a general fiber of $P$,  and  $\mathcal H = \Phi^* \mathcal H'$.

Our second main result describes the structure of Riccati foliations having irregular singularities.

\begin{THM}\label{THM:C}
Suppose that $X$ is a projective manifold, and  $(P,\mathcal H)$ is a Riccati foliation over $X$.
If $\mathcal H$ is irregular then at least one of the following assertions holds true.
\begin{enumerate}
\item Maybe after passing to a (possibly ramified) two-fold covering, the Riccati foliation
 $\mathcal H$ is defined by a closed rational $1$-form.
\item The Riccati foliation $(P,\mathcal H)$ factors through a curve.
\end{enumerate}
\end{THM}

The proof of Theorem \ref{THM:C}  relies on similar ideas as those used in the proof of Theorem \ref{THM:A}, on an infinitesimal criterion for the existence
of fibrations due to Neeman (Theorem \ref{T:Neeman}) and on the semi-local study of Riccati foliations
at a neighborhood of irregular components of the polar divisor (Section \ref{S:irregular}).

\subsection{Transversely projective foliations}
Our main goal, the description of the structure of transversely projective foliations, is achieved by combining Corollary \ref{COR:B} and Theorem \ref{THM:C}. To this end, recall that for a polydisk Shimura modular orbifold $\mathfrak H$
and any of its tautological representations $\rho$, Deligne extension of the associated local system provides a logarithmic flat connection; denote by $(\mathfrak H\times_\rho\P^1,\mathcal H_\rho)$ the induced Riccati foliation.

\begin{THM}\label{THM:D}
Let $\mathcal F$ be a codimension one transversely projective
foliation on a projective manifold $X$.
Then at least one of the following assertions
holds true.
\begin{enumerate}
\item There exists a generically finite Galois morphism $f: Y \to X$ such that $f^* \mathcal F$ is defined by a closed rational $1$-form.
\item There exists a rational dominant map $f: X \dashrightarrow S$ to a ruled surface $\pi:S\to C$, and a Riccati foliation
$\mathcal H$ defined on $S$ (i.e. over the curve $C$) such that $\mathcal F= f^* \mathcal H$.
\item There exists a polydisk Shimura modular orbifold $\mathfrak H$ and a rational map $f: X \dashrightarrow \mathfrak H\times_\rho\P^1$ towards one of its tautological Riccati foliations such that $\mathcal F= f^* \mathcal H_\rho$.
\end{enumerate}
\end{THM}

\begin{remark}In the third item, we note (see \cite{ErwanFred}) that, after blowing-up the ambient,  $\mathcal H_\rho$
(resp. $\mathcal F$) is locally defined at singular points by a $1$-form of the type
$$\sum_{i=1}^k\lambda_i\frac{dx_i}{x_i}, \ \ \ 1\le k\le n,\ \lambda_i\in\mathbb R_{>0}$$
for local coordinates $(x_1,\ldots,x_n)$ (where $n$ is the dimension of the ambient space).
In particular, for instance for $X$ a surface, if $\mathcal F$ has (maybe after blowing-up) 
a hyperbolic singular point, a saddle-node or a non linearizable saddle, then we are in the first two items.
\end{remark}

There are previous results on the subject \cite{Scardua}  and on the
neighboring subject of transversely affine foliations  (\cite{CamachoScardua}, \cite{CerveauSad}, \cite{GaelJorge}).
With the exception of \cite{GaelJorge}, all the other works impose strong restrictions
on the nature of the singularities of the foliation. Our only hypothesis is the projectivity of the ambient manifold.

Theorem \ref{THM:D} also answers
a question left open in \cite{GaelJorge}. There, a similar  classification for transversely affine foliations is established for foliations on projective manifolds with zero first Betti number. Theorem \ref{THM:D}
gives the analogue classification  for arbitrary projective manifolds, showing that the hypothesis on the first Betti number is not necessary.

\subsection{Flat meromorphic $\mathfrak{sl}_2$-connections}

A meromorphic rank $2$ connection $(E,\nabla)$ on a projective manifold $X$
is the datum of a rank $2$ vector bundle $E$ equipped with a $\mathbb C$-linear morphism  $\nabla:E\to E\otimes\Omega^1_X(D)$
satisfying Leibniz rule 
$$\nabla(f\cdot s)=f\cdot \nabla(s)+df\otimes s\ \ \ \text{for any section }s\text{ and function }f.$$
Here $D$ is the polar divisor of the connection $\nabla$. 
The connection $\nabla$  is {\it flat} when the curvature vanishes, that is $\nabla\cdot\nabla=0$,
meaning that it has no local monodromy outside the support of $D$; we can therefore define its monodromy representation.
{\it Throughout the text, all connections are assumed to be flat.}

When  $\det(E)=\mathcal O_X$ and the trace connection $\mathrm{tr}(\nabla)$ is the trivial connection on $\mathcal O_X$, 
we say that $(E,\nabla)$ is a {\it $\mathfrak{sl}_2$-connection}. In particular, its monodromy representation 
takes values into $\SL(\C)$.

We will say that any two connections $(E,\nabla)$ and $(E',\nabla')$ are {\it birationally equivalent} when
there is a birational bundle transformation $\phi:E\dashrightarrow E'$ that conjugates the two operators $\nabla$ and $\nabla'$. In other words, $\mathcal O_X(\star D)\otimes_{\mathcal O_X}E\simeq \mathcal O_X(\star D)\otimes_{\mathcal O_X}E'$ for some reduced divisor $D$ in $X$ 
and the respective differential operators coincide.
Keep in mind that the polar divisor might be not the same for $\nabla$ and $\nabla'$. The connection $(E,\nabla)$ is called regular if local $\nabla$-horizontal sections have moderate growth near the polar divisor 
$(\nabla)_\infty$; equivalently, the restriction of $(E,\nabla)$ to a general (complete) curve $C\subset X$
is birationally equivalent to a connection having only simple poles (see \cite[chap.II, Definition 4.2]{Deligne} for details). When the polar divisor $(\nabla)_\infty$ has only normal crossings, then regularity is equivalent
to be birationally equivalent to a logarithmic connection, in particular having only simple poles. 
A connection $(E,\nabla)$ is said irregular if it is not regular.

We will say that $(E,\nabla)$ and $(E',\nabla')$ are {\it projectively equivalent} if
the induced $\P^1$-bundles coincide $\P (E)=\P (E')$, and if moreover $\nabla$ and $\nabla'$ induce the same projective connection $\P (\nabla)=\P (\nabla')$; equivalently, they induce the same Riccati foliation on the total space of $\P (E)$. This is equivalent to say that $(E',\nabla')=(E,\nabla)\otimes(L,\delta)$
for some flat meromorphic rank $1$ connection $(L,\delta)$ on $X$.

We provide a structure theorem for connections up to the combination of projective and birational equivalence.
In order to settle our result, we note that any meromorphic rank $2$ connection is (projectively birationally) equivalent
to a $\sl$-connection (see remark \ref{R:generalriccati}).
A combination of Corollary \ref{COR:B} and Theorem \ref{THM:C} yields:

\begin{THM}\label{THM:E}Let $(E,\nabla)$ be a flat meromorphic $\sl$-connection 
on a projective manifold $X$.
Then at least one of the following assertions holds true.
\begin{enumerate}
\item There exists a generically finite Galois morphism $f: Y \to X$ such that $f^*(E,\nabla)$ is  projectively birationally equivalent to one of the following connections defined on the trivial bundle:
$$\nabla=d+\begin{pmatrix}\omega&0\\ 0&-\omega\end{pmatrix}\ \ \ \text{or}\ \ \ d+\begin{pmatrix}0&\omega\\ 0&0\end{pmatrix}$$
with $\omega$ a rational closed $1$-form on $X$.
\item There exists a rational map $f: X \dashrightarrow C$ to a curve and a meromorphic connection $(E_0,\nabla_0)$ on $C$
such that $(E,\nabla)$ is projectively birationally equivalent to $f^*(E_0,\nabla_0)$.
\item The $\mathfrak{sl}_2$-connection $(E,\nabla)$
has at worst regular singularities and there exists a rational map $f: X \dashrightarrow \mathfrak H$ which projectively factors the monodromy  through one of the tautological representations of a polydisk Shimura modular orbifold $\mathfrak H$. In particular,
the monodromy representation of $(E,\nabla)$ is quasi-unipotent at infinity, rigid, and
Zariski dense.
\end{enumerate}
\end{THM}

In particular, when $(E,\nabla)$ is irregular, only the former two cases occur. As mentioned in the abstract, the connection  projectively factors through a curve whenever it has non trivial Stokes. 

\begin{remark}
As we are working in the meromorphic setting, it is no more relevant to consider orbicurves instead of curves in item (2).
\end{remark}

\subsection{Structure of the paper}
The paper is divided in two parts, with the first independent of the second. In the first part,  Sections \ref{S:Fibration} and \ref{S:Factor}, we recall some results on the existence of fibrations which will be used throughout the paper, and present the proof of Theorem \ref{THM:A}. In this first part we avoided using  foliation theory aiming at a wider audience. The second part deals with the irregular case and is organized as follows.
Section \ref{S:basic} presents foundational  results about Riccati foliations, most of them borrowed from \cite{RH} and \cite{GaelJorge}. In particular, we show how to reduce the general problem to the surface case. Section \ref{S:polar} describes the local structure of a Riccati foliation along its polar divisor (on a surface);
we state there, in Theorem \ref{THM:ReductionRiccati}, a  reduction of singularities in the spirit of Sabbah's ``good formal model'' for meromorphic connections by using blowing-ups and ramified covers; the proof of this technical result is postponed to Section \ref{SecTransvProj},
and it uses tools from the theory of transversely projective foliations, like classification of their reduced singular points following  
\cite{BerthierTouzet,FredTransvProj}.  Section \ref{S:irregular} analyzes the structure of reduced Riccati foliations 
over surfaces 
at a neighborhood of a connected component of its irregular divisor, showing in particular  the existence of flat tranverse coordinates. This will be essential to produce the fibration in the absence of rich monodromy.
Section \ref{S:structure} contains the proofs of Theorems \ref{THM:C}  and  \ref{THM:E}.
Section \ref{SecTransvProj} deals with transversely projective foliations, including 
the proof of  Theorem \ref{THM:D}, and the fact that it is actually equivalent to our structure result 
on $\sl$-connections.  Section \ref{SecExamples} presents a series of examples
underlining the sharpness of our results. Finally in an appendix we prove a result reminiscent of Sabbah's good formal models in the context of Riccati foliations.

\subsection{Thanks} The authors warmly thank anonymous referees for pointing out several
uncorrect points and for giving many suggestions improving the presentation of the paper.

\section{Existence of fibrations}\label{S:Fibration}

In this section we collect three results about the existence of fibrations on projective manifolds which will be used in the sequel.
For the  proof of Theorem \ref{THM:A}  all we will need is Theorem \ref{T:Totaro} below, which is  due to  Totaro \cite{Totaro}. Toward the end of the paper (proof of Theorem \ref{THM:C}), we will
make use of the two other  results below.

\begin{thm}\label{T:Totaro}
Let $X$ be a projective manifold and
 $D_1, D_2, D_3$ $3$ connected effective divisors which are pairwise disjoint and whose
Chern classes lie in a line inside of  $H^2(X, \mathbb R)$. Then there exists a non constant morphism
$f: X\to C$  to a smooth curve $C$ with connected fibers which contracts every divisor numerically proportional to $D_1$  to a point. In particular, the $D_i$'s are rational multiple of fibers of $f$.
\end{thm}

The original proof studies the restriction map $H^1(X,\mathbb Q) \to H^1(\hat{D}_1,\mathbb Q)$, where $\hat D_1$ is the disjoint union
of desingularizations of the irreducible components of $D_1$.
When it is injective, this map leads to a divisor linearly equivalent to zero in the span of $D_2, \ldots, D_r$
which defines the  fibration. Otherwise, the fibration is constructed as a quotient of the Albanese map
of $X$. An alternative proof, based on properties of some auxiliary foliations is given in \cite{jvpJAG}.
It goes as follows:  given two divisors with proportional Chern classes, one constructs a logarithmic $1$-form with poles on
these divisors and purely imaginary periods. The induced foliation, although not by algebraic leaves in general, admits
a non-constant real first integral. Comparison of  the leaves of two of these foliations coming from three pairwise disjoint
divisors with proportional Chern classes reveals
that they are indeed the same foliation. The proportionality factor of the corresponding two logarithmic $1$-forms gives, after Stein factorization, the
sought fibration. For details,  see respectively \cite{Totaro, jvpJAG}.

In general, two disjoint divisors with same Chern classes are not fibers of a fibration. Indeed, if $L$
is a non-torsion flat line-bundle over a projective curve $C$ then the surface $S=\P (  L \oplus \mathcal O_C)$ admits two homologous disjoint curves, corresponding to the inclusions of $L$ and $\mathcal O_C$ in $L \oplus \mathcal O_C$, which are not fibers of a fibration. The point is that the normal bundle of these sections are $L^*$ and $ L$. Nevertheless, if the normal bundle of one of the effective divisors is torsion then we do have a fibration containing them as fibers. This example is extracted from \cite[p. 613]{Totaro}.

\begin{thm}\label{T:Totarobis}
Let $X$ be a projective manifold and
 $D_1, D_2$  be connected effective divisors which are pairwise disjoint and whose
Chern classes lie in a line  of  $H^2(X, \mathbb R)$.
If $\mathcal O_X(D_1)\vert_{D_1}$ is torsion then there exists a non constant map
$f: X\to C$  to a smooth curve $C$ with connected fibers which
maps the divisors $D_i$ to points.
\end{thm}

A proof of this result is presented in Section \ref{S:Totarobis} below.
There is also  an infinitesimal variant of
Theorem \ref{T:Totaro} due to Neeman \cite[Article 2, Theorem 5.1]{Neeman} where instead of three divisors we have only one divisor with torsion normal bundle and  constraints on a infinitesimal neighborhood of order bigger
than the order of the normal bundle.
The formulation below is due to Totaro, see \cite[paragraph before the proof of Lemma 4.1]{Totaro2}, and  the proof we present  is an
adaptation of Neeman's proof. For an alternative proof see \cite[Section 3]{Leaves}.

\begin{thm}\label{T:Neeman}
Let $X$ be a projective manifold and
 $D$  be a connected effective divisor. Suppose that  $\mathcal O_X(D)\vert_{D}$ is torsion of order $m$. If
\[
\mathcal O_X(mD)\vert_{(m+1)D}  \simeq \mathcal O_{(m+1)D}
\]
then there exists a nonconstant morphism
$f: X\to C$  to a smooth curve $C$ with connected fibers which
maps the divisor $D$ to a point.
\end{thm}
\begin{proof}
Let $I = \mathcal O_X(-D)$ be the defining ideal of $D$. Consider the diagram
\begin{center}
\begin{tikzpicture}
 \matrix (m) [matrix of math nodes,row sep=2em,column sep=2.7em,minimum width=1em]
 {
  0 & \mathcal O_X & \mathcal O_X(mD) & \mathcal O_{X}(mD) \otimes \frac{\mathcal O_X}{I^m}  &  0 \\
  0 &  \mathcal O_D & \mathcal O_X(mD) \otimes \frac{\mathcal O_X}{I^{m+1}} & \mathcal O_{X}(mD) \otimes \frac{\mathcal O_X}{I^m}  &  0 \\};
 \path[-stealth]
   (m-1-1) edge (m-1-2)
 (m-1-2) edge (m-1-3)
  (m-1-3) edge (m-1-4)
  (m-1-4) edge (m-1-5)
   (m-2-1) edge (m-2-2)
   (m-2-2) edge (m-2-3)
   (m-2-3) edge  (m-2-4)
   (m-2-4) edge (m-2-5)
   (m-1-2) edge   (m-2-2)
   (m-1-3)     edge   (m-2-3)
  (m-1-4)  edge node [left] {$\wr$} (m-2-4)  ;
\end{tikzpicture}
\end{center}
deduced from the standard exact sequence
\[
0 \to \mathcal O_X(-mD) \longrightarrow \mathcal O_X \longrightarrow \frac{\mathcal O_X}{I^m} \to 0 \, .
\]

In cohomology we get the diagram
\begin{center}
\begin{tikzpicture}
 \matrix (m) [matrix of math nodes,row sep=2em,column sep=1.5em,minimum width=0.8em]
 {  & H^0(X,\mathcal O_{X}(mD)\otimes \frac{\mathcal O_X}{I^{m}}) & H^1(X,\mathcal O_X)   \\
  H^0((m+1)D,   \mathcal O_X(mD) \otimes \frac{\mathcal O_X}{I^{m+1}}) &  H^0(mD,    \mathcal O_X(mD) \otimes \frac{\mathcal O_X}{I^{m}})  &  H^1(D ,\mathcal O_D) \\};
 \path[-stealth]
 (m-1-2) edge (m-1-3)

   (m-2-1) edge (m-2-2)
   (m-2-2) edge (m-2-3)

   (m-1-2)     edge  node [left] {$\wr$} (m-2-2)
  (m-1-3)  edge (m-2-3)  ;
\end{tikzpicture}
\end{center}
If $\mathcal O_X(mD)\vert_{(m+1)D} \simeq \mathcal O_{(m+1)D}$ then
$\mathcal O_X(mD)\vert_{mD} \simeq \mathcal O_{mD}$ and
$$1 \in H^0(mD,\mathcal O_{mD}) = H^0(mD,    \mathcal O_X(mD) \otimes \frac{\mathcal O_X}{I^{m}})  $$ belongs to the image of the map in the lower left corner. Exactness of the bottom row
implies that $1 \in H^0(mD,\mathcal O_{mD}) $  is mapped to zero in $H^1(D,\mathcal O_D)$.

According to the diagram, the morphism $H^0(mD,\mathcal O_{mD})  \to H^1(D,\mathcal O_D)$ factors through $H^1(X,\mathcal O_X) \to H^1(D,\mathcal O_D)$.

If $1 \in H^0(mD,\mathcal O_{mD}) $ is mapped to zero in $H^1(X,\mathcal O_X)$  then we deduce from the first row of the first diagram above that $h^0(X,\mathcal O_X(mD))\ge 2$. Moreover, $H^0(X,\mathcal O_X(mD))$ contains a section nowhere vanishing on $D$.  Therefore
$mD$ moves in a pencil without base locus and we have the sought fibration.

If instead   $1 \in H^0(mD,\mathcal O_{mD})$ is mapped to a nonzero element in $H^1(X,\mathcal O_X)$ then $H^1(X,\mathcal O_X) \to H^1(D,\mathcal O_D)$ is not injective. Thus the same
holds true for the map $H^1(X,\mathcal O_X) \to \oplus H^1(D_i,\mathcal O_{D_i})$ where $D_i$ are the irreducible components of $D$. It follows that $\oplus \Alb(D_i)$  (direct sum of the Albanese varieties) do not dominate  $\Alb(X)$. By the second part of the proof of Theorem 2.1 in \cite{Totaro},  the morphism
\[
X \longrightarrow \frac{\Alb(X)}{ \oplus \Alb(D_i) }
\]
contracts $D$, and is non constant. It follows (cf. \cite{Neeman} or \cite{Totaro}) that the image is a curve and we get the sought fibration
as the Stein factorization of this morphism.
\end{proof}

\subsection{Proof of Theorem \ref{T:Totarobis}}\label{S:Totarobis}
By assumption the effective divisors $D_1$ and $D_2$ have proportional Chern classes.
Therefore, there exists non-zero positive integers $a_1, a_2$ such that the line bundle $\mathcal L = \mathcal O_X(a_1 D_1 - a_2 D_2)$ lies
in $\Pic^0(X)$. If $\mathcal L$ is a torsion element of $\Pic^0(X)$ then there exists a rational function $g: X \to \mathbb P^1$ with
zero set supported on $D_1$ and polar set supported on $D_2$. We can take $f$ as the Stein factorization of $g$.

Suppose now that $\mathcal L$ is not a torsion line-bundle. The restriction $\mathcal L\vert_{D_1}$ of $\mathcal L$ to $D_1$ is isomorphic to $\mathcal O_{D_1}(a_1 D_1)$ and according to our
assumptions is a torsion line-bundle. Hence for some integer $m\neq 0$, $\mathcal L^{\otimes m}$ is in the kernel of the restriction   morphism $\Pic^0(X)\to \Pic^0(D_1)$. Since we are assuming that $\mathcal L$ is not a torsion line-bundle, it follows that the restriction morphism $H^1(X,\mathcal O_X) \to H^1(D_1,\mathcal O_{D_1})$ is  not injective.    We conclude as in the proof of Theorem \ref{T:Neeman}. \qed

\section{Factorization of representations}\label{S:Factor}

\subsection{Criterion for factorization}
We apply Theorem \ref{T:Totaro} to establish a criterion
for the factorization of representations of quasiprojective fundamental groups.
In the statement below, we have implicitly fixed an ample
divisor $A$ in $X$ and  we consider the bilinear pairing defined in $NS(X)$, the N\'eron-Severi group of
$X$,  defined
by $(E,D) = E \cdot D \cdot A^{n-2}$ where $n = \dim X$. According to
Hodge index Theorem  this bilinear form has signature $(1,\rank NS(X)-1)$.

\begin{thm}\label{T:criteriofactor}
Let $X$ be a  projective  manifold,  $D$ be a reduced simple normal crossing divisor in $X$,
and  $\rho:\pi_1(X-D) \to G$ be a representation to a simple linear algebraic group $G$ \footnote{ As the terminology may vary, we precise that a  linear algebraic group G is \textit{simple} if it is not commutative and has no normal
algebraic subgroups (other than 1 and $G$), and it is quasi-simple if its centre Z is finite
and $G/Z$ is simple. } with Zariski dense image.
Suppose there exists $E$, a connected component of the support $D$ with irreducible
components $E_1, \ldots, E_k$,
and an (analytic) open subset $U \subset X$ containing $E$  such that the restriction of $\rho$
to $\pi_1(U - E)$ has solvable image. Then
either the intersection matrix $(E_i, E_j)$ is negative definite or
the representation factors through an orbicurve.
\end{thm}
\begin{proof}
If the intersection matrix $(E_i,E_j)$ is negative definite then there is nothing to prove. Throughout we 
will assume that $(E_i,E_j)$ is not negative definite.

Let $S$ be the Zariski closure of $\rho(\pi_1(U-E))$ in $G$. Since $G$ is simple while
$\rho(\pi_1(U-E))$ is solvable,
we have that $S \neq G$. Let $\mu$ be the derived length of $S$  and  choose an  element
$\gamma$ of the $(\mu+1)$-th derived  group of $\pi_1(X-D)$ such that $\rho(\gamma) \neq \id$.

We  apply Malcev's Theorem (any finitely generated subgroup of $G$ is residually finite) \cite{malcev} to  obtain
a morphism $\varrho : \pi_1(X-D) \to \Gamma$ to a finite group $\Gamma$ such that
$\varrho(\gamma)\neq \id$.  This choice of $\varrho$ implies that the derived length of $\varrho(\pi_1(X-D))$ is at least $\mu+2$.
Since the derived length of $S$ is $\mu$, it follows that $c = [ \varrho(\pi_1(X-D)) : \varrho(\pi_1(U-E)) ]$ is at least
$3$.

Let us now consider the covering $p: X_{\varrho} \to X$, ramified along $D$, determined by $\varrho$. Over $X-D$, $p\vert_{p^{-1}(X-D)}$ is the \'etale
covering of $X-D$ determined by $\ker \varrho$. In particular $p^{-1}(X-D)$ is a smooth quasi-projective manifold. In general, $X_{\varrho}$ is projective but not necessarily smooth. Notice that by construction $p^{-1}(E)$  has at least
$c=[ \varrho(\pi_1(X-D)) : \varrho(\pi_1(U-E))] \ge 3$ distinct connected components in $X_{\varrho}$.

To be able to apply Hodge index Theorem, we now proceed to {\it desingularize} $X_{\varrho}$. In order to keep track 
of the intersection matrix $(E_i,E_j)$, instead of applying Hironaka's Theorem
we will construct a smooth variety $Y$ together with a ramified covering  $\phi: Y\to X$ which factors through $p$.

To each irreducible component $D_i$ of $D$ let $m_i$ be the order of $\varrho(\gamma_i)$  on a short loop $\gamma_i$ around $D_i$.
Let $\kappa: \tilde X \to X$ be a  finite ramified  covering of $X$, with $\tilde X$ smooth,  such that $\kappa^*(D_i) = m_i \tilde{D_i}$
where $\tilde{D_i}$ is a smooth and irreducible hypersurface and  $\tilde D= \sum \tilde{D_i}$ is a simple normal
crossing divisor.  For the existence of $\kappa$ with the above properties, we can use Kawamata coverings (see \cite[Proposition 4.1.12]{Lazarsfeld}).
If we denote the ramification divisor of $\kappa$ by $R$ and its image under $\kappa$ by $\Delta$ then we can also assume that $\Delta + D$ is a simple normal crossing divisor on $X$.

The composition $ \varrho\circ \kappa_* : \pi_1(\tilde X - \tilde D) \to \Gamma$ sends
short loops around the irreducible components of $\tilde D$ to the identity of $\Gamma$ and therefore induces a representation
$\tilde \varrho : \pi_1(\tilde X) \to \Gamma$.
Let $\pi:Y \to \tilde X$ be the \'{e}tale covering of $\tilde X$ determined by the kernel of $\tilde \varrho$. It is a projective
 manifold endowed with a ramified covering  $\phi = \kappa\circ \pi: Y \to X$ 
 which factors through $p:X_{\varrho} \to X$  as
 wanted. Consequently $\phi^{-1}(E)$ has (at least) as many connected components as $p^{-1}(E)$.

Let $B$ be one of the connected components of $\phi^{-1}(E)$
with irreducible components $B_1, \ldots, B_{\ell}$.
Consider the ample divisor $A':=\phi^* A$ on $Y$ and define  $(\cdot,\cdot)_Y$ using it.
Notice that 
\[
(\phi^* C_1 , \phi^* C_2)_Y= \phi^* C_1 \cdot \phi^* C_2 \cdot (\phi^* A)^{n-2} =  \deg(\phi) \cdot (C_1,C_2)
\]
for any divisors $C_1, C_2$ in $X$.
Since we are assuming that $(E_i, E_j)$ is not negative definite, there exists an effective divisor $F$
supported by $B$ with $(F,F)_Y \ge 0$.
 In fact, there exists such divisor for each connected component of $\phi^{-1}(E)$,
therefore at least three. Hence, we can produce $F_1, F_2, F_3$ pairwise disjoint
effective divisors with connected supports on $Y$ satisfying $(F_i,F_i)_Y  \ge 0$ and $ (F_i,F_j)_Y=0$.
Since the signature of the quadratic form $(\cdot,\cdot)_Y$
is $(1,\rank NS(Y)-1)$, we  deduce that all the three divisors have proportional Chern classes; moreover, $(F_i,F_i)_Y  = 0$.
Notice also that either $F_i \cap B = \emptyset$
or the support of $F_i$ coincides with $B$. Indeed, if $F$ is an effective divisor  with support contained in $B$, but not equal to $B$, 
then we can choose an irreducible component of $B$, say $C$, not contained in the support of $F$ but intersecting it. Therefore $(C,F)_Y >0$ and, provided that $(F,F)_Y \ge 0$,
for $k$ large enough $C+kF$ is an effective divisor with $(C+kF,C+kF)_Y = (C,C)_Y + 2k (C,F)_Y+k^2 (F,F)_Y>0$. But,
doing so, we would produce new disjoint $F_i$'s with $(F_i,F_i)_Y>0$, contradicting 
Hodge index Theorem.

We can apply Theorem \ref{T:Totaro} to ensure the existence
of a curve $\Sigma$ and a nonconstant morphism with irreducible general fiber $g:  Y \to \Sigma$  such that the divisors $F_1, F_2,$ and $F_3$ are
multiples of  fibers of $g$. The morphism $g$ is proper and open, thus all the other connected components of $\phi^{-1}(D)$ are mapped by $g$ to points. Let us denote by $\tilde\rho:=\rho\circ\phi_*:\pi_1(Y-\phi^*D)\to G$ the lifted representation.
We want first to prove that $\tilde\rho$ factors through $g$. Note that $\tilde\rho$ also has  Zariski dense image in $G$.

Let $U \subset Y-\phi^*D$ be a Zariski open subset such that the restriction of $g$ to $U$ is a
smooth and proper fibration, thus locally trivial in the $C^{\infty}$ category, over $\Sigma^o= g(U)$.
Let also {$F$ be a fiber of $g_{\vert U}$ and} $H$ be  the Zariski closure in $G$ of $\tilde\rho(\pi_1(F))$, and consider
the following diagram 
\begin{center}
\begin{tikzpicture}
  \matrix (m) [matrix of math nodes,row sep=2em,column sep=4em,minimum width=2em]
  {
   1 & H & G &   &   \\
    & \pi_1(F) & \pi_1 (U) & \pi_1(\Sigma^o ) & 1\\};
  \path[-stealth]
    (m-1-1) edge (m-1-2)
  (m-1-2) edge (m-1-3)
    (m-2-2) edge (m-2-3)
    (m-2-3) edge node [above] {$g_*$} (m-2-4)
    (m-2-4) edge (m-2-5)
    (m-2-2) edge node [left] {$\tilde\rho$}  (m-1-2)
    (m-2-3)     edge node [left] {$\tilde\rho$} (m-1-3) ;
\end{tikzpicture}
\end{center}

From this, it is clear that the image of $\pi_1(F)$ is normal in $\pi_1(U)$, and since normality is a (Zariski) closed condition,
we deduce that $H$ is a normal subgroup of $G$. Since $G$ is simple, we conclude that
$H$ must be trivial, i.e. the restriction of $\tilde\rho$ to $U$ factors through the curve $\Sigma^o$.  Following the proof of \cite[Lemma 3.5]{Corlette-Simpson}
we see that this suffices to obtain the factorization through an orbicurve.
Still denote by $g:Y\to \Sigma$ the factorization morphism.

We will now prove that the morphism $g:Y\to \Sigma$ descends to a morphism $f:X\to C$
that factors the initial representation $\rho$. Here, we follow the argument of \cite[Lemma 3.6]{Corlette-Simpson}.
By Stein factorization Theorem, we can assume that $g:Y\to \Sigma$ has connected fibers.
Assume also $\phi:Y\to X$ is Galois, with Galois group $\Lambda$; if not, replace $Y$ by some finite Galois covering $\tilde\phi:\tilde Y\stackrel{\varphi}{\to}Y\stackrel{\phi}{\to}X$
(choose a finite index normal subgroup of $\pi_1(X)$ that contains the subgroup defining $\phi$).
We note that $Y$ might become singular, but this does not matter in what follows.
For any $\gamma \in \Lambda$, we want to prove that $\gamma$ permutes $g$-fibers, i.e. 
for a generic fiber $F$ of $g$, then  $g\circ\gamma(F)$ is a point.

Aiming at a contradiction assume  $g\circ\gamma(F)$ is not a point. Then $g_{\vert\gamma(F)}:\gamma(F)\to\Sigma$ is surjective and, since the representation $\tilde\rho$
factors through the curve $\Sigma$, it follows that the representation $\tilde\rho$ has Zariski dense image
in restriction to $\gamma(F)$. Strictly speaking, in order to define the restriction $\tilde\rho$ to $\gamma(F)$,
we have to move the base point of the fundamental group (that we have omitted so far) to put it into $\gamma(F)$
and the restriction depends on the way we do this, but different choices lead to conjugated subgroups and the property of  being Zariski dense is invariant by conjugation.
Once this has been done, the image $H$ of $\tilde\rho_{\vert\gamma(F)}$ must be of finite index
in the image of the factorizing representation $\pi_1(\Sigma)\to G$. By hypothesis $G$ is simple and therefore $H$ is Zariski dense.
On the other hand, since $\tilde\rho$ comes from a representation on $X=Y/\Lambda$, it follows that
$\tilde\rho$ must be trivial in restriction to $\gamma(F)$, since it is in restriction to $F$. 
This contradiction shows that $\Lambda$, the Galois group of $\phi$, permutes the fibers  of $g:Y\to \Sigma$, and thus permutes
the points of $\Sigma$ (connectedness of fibers). We get a morphism from $X$ to the orbicurve $C=\Sigma/\Lambda$
through which $\rho$ factors.\end{proof}

\begin{remark}
We can avoid the factorization through an orbicurve by instead restricting the factorization to a Zariski open subset of $X- D$.
In the opposite direction, if we allow one dimensional Deligne-Mumford stacks with general point having non trivial stabilizer as targets of the factorization, then we can replace simple linear algebraic groups by quasi-simple linear algebraic groups in the statement, since our proof
shows that for $G$ quasi-simple, there exists a fibration such that the Zariski closure of the images of fundamental groups of fibers of $f$ under $\rho$ are finite.
\end{remark}

\subsection{Rank-two representations  at  neighborhoods of divisors}

\begin{prop}\label{P:topsolv}
Let $X$ be a complex manifold,   $D$ a reduced and simple normal crossing divisor in $X$, and $\rho:\pi_1(X-D) \to \SL(\C)$  a representation.
Let $E$ be a connected divisor with support contained in $D$ such that
for each irreducible
component $E_i$ of $E$ and any short loop $\gamma_i$ turning around $E_i$, the element  $\rho(\gamma_i)$ does not lie in the center of $\SL(\C)$, i.e. $\rho(\gamma_i)$ is distinct from $\pm \Id$.
Then  there exists an open subset $U \subset X$ containing $E$  such that the restriction of $\rho$
to $\pi_1(U - D)$ has solvable image.
\end{prop}
\begin{proof}
Let $E_1, \ldots, E_k$ be the irreducible components of $E$ and $\gamma_1, \ldots, \gamma_k$ be
short loops turning around them.  We will denote the set of smooth points of $D$ in $E_i$ by
 $E_i^{\circ}$, i.e. {$E_i^{\circ}=E_i - \cup_{j\not=i}(E_j\cap E_i)$}.

Suppose first that $\pm \rho(\gamma_1)$ is unipotent. Since, by hypothesis, it is different from the identity,
its action on $\C^2$ leaves invariant a one-dimensional subspace $L$. If $U_1$ is a small tubular  neighborhood of
$E_1$ and $U_1^{\circ} = U_1 - D$ then $U_1^{\circ}$ has the homotopy type of a $S^1$-bundle over $E_1^{\circ}$
and therefore the subgroup generated by $\gamma_1$ in $\pi_1(U_1^{\circ})$ is normal. It follows that
every $\gamma \in \pi_1(U_1^{\circ})$ also leaves $L$ invariant. It follows that the rank two local system induced
by $\rho$ admits a unique rank one local subsystem determined by $L$ on $U_1^{\circ}$.

To analyze what happens  at a non-empty  intersection $E_1 \cap E_j$,  we can assume that both $\gamma_1$ and $\gamma_j$ have base points
near $E_1\cap E_j$. Thus  {$\gamma_1$ commutes with
$\gamma_j$,   both $\rho(\gamma_1)$ and $\rho(\gamma_j)$ are  unipotent, and they both leave $L$ invariant}.
Thus the rank one local subsystem determined by $L$ on $U_1^{\circ}$ extends to a rank one local subsystem on $U_1^{\circ} \cup U_j^{\circ}$.
Repeating the argument above for the other irreducible components $E_2, \ldots, E_k$,
we deduce the existence of a neighborhood $U$ of $E$ such that $\rho(\pi_1(U-D))$ is
contained in a Borel subgroup of $\SL(\C)$.

Similarly if $\rho(\gamma_1)$ is semi-simple, then the same holds true for every $\gamma_i$. Moreover,
the representation now leaves invariant the union of two linear subspaces $L_1$ and $L_2$ (but does not necessarily
leave invariant any of the two). We deduce that the image of $\rho$ restricted to a neighborhood of $E$ minus $D$
is contained in an extension of $\mathbb Z/2\mathbb Z$ by $\C^*$.
\end{proof}

\begin{remark}
The proof above is very similar to  the  proof of \cite[Lemma 4.5]{Corlette-Simpson}.
\end{remark}

\begin{prop}\label{P:naocontrai}
Let $X, D, E,U$ and $\rho$ be as in Proposition \ref{P:topsolv}.
Assume also that every short loop $\gamma$ turning around {an  irreducible component of $D - E$ which intersects $E$}
has monodromy in the center of $\SL(\C)$. If the intersection matrix of $E$
is invertible, then the restriction of $\rho$ to $\pi_1(U-D)$ is quasi-unipotent at  the irreducible
components of $E$.
\end{prop}
\begin{proof}
Let $(F,\nabla)$ be a rank two vector bundle over $U$ with a flat logarithmic connection whose monodromy
is given by $\rho$ (see \cite[Proposition  5.4, p.94]{Deligne}). Since the monodromy is solvable, around each point of $U-D$ we have one or two
subbundles of $F$ which are left invariant by $\nabla$. Modulo passing to a double covering of
$U-D$ if necessary, we can assume that $(F,\nabla)$  is reducible, i.e. we have a subbundle $F_1\subset F$ and a logarithmic connection $\nabla_1$ on $F_1$ such that
$\nabla_1 = \nabla\vert_{F_1}$.
The  monodromy of $\nabla_1$ on a loop $\gamma$ around irreducible components of $E$ equal   to
one of the eigenvalues of $\rho(\gamma)$, say $\lambda_{\gamma}$.
If $\gamma_i$ is a short loop around an irreducible component $E_i$ of $E$ then the residue of $\nabla_1$ along
$E_i$ satisfies
\[
\exp( 2 \pi i \Res_{E_1}(\nabla_1) )  =  \lambda_{\gamma_1}  \, .
\]
By the residue formula we can write
\[
  c_1(F_1) = \sum \Res_{E_1}(\nabla_1) E_i + \sum \Res_{D_j}(\nabla_1) D_j
\]
where $D_1, \ldots, D_s$ are the irreducible components of $D$ intersecting $E$ but not contained in it.
Since the eigenvalues around $D_j$ are $\pm 1$, we have that $\Res_{D_j}$ is a half-integer.
 Since $c_1(F_1)$ lies in $H^2(U, \mathbb Z)$, we have that for every $k$
\[
 (c_1(F_1) , E_k) = \sum \Res_{E_i}(\nabla_1) (E_i  , E_k) +  \sum \Res_{D_j}(\nabla_1) (D_j , E_k)
\]
is an integer. Therefore the vector $v = (\Res_{E_1}(\nabla_1), \ldots, \Res_{E_k}(\nabla_1))^t$ satisfies
a linear equation of the form $A \cdot v = b$ with $A=(E_i ,  E_j)$ and $2b \in \mathbb Z^k$.
Since $A$ has integral coefficients and is negative definite, it follows that $v$ is a rational vector. Therefore
the restriction of $\rho$ to $U-D$ is quasi-unipotent at the irreducible components of $E$.
\end{proof}

\begin{remark}
The proof above is reminiscent of Mumford's computation \cite{Mumford} of the homology of the plumbing of a contractible divisor
on a smooth surface.
\end{remark}

\subsection{Proof of Theorem \ref{THM:A}}
 Let $X$ be a projective manifold and $D\subset X$ a simple normal crossing hypersurface such that $\X0 =X-D$. Let $\rho: \pi(X-D) \to \SL(\C)$ be a Zariski dense representation which is not quasi-unipotent at infinity.
Let $E$ be a connected divisor with support contained in $|D|$ \footnote{If $H$ is a divisor, $|H|$ stands for the \textit{support} of H} such that $\rho(\gamma) \neq \pm \Id$ for every small  loop around an irreducible component of $E$, and $\rho(\gamma)$ is not quasi-unipotent for at least one small loop. If $E$ is maximal with respect to these properties, Proposition \ref{P:topsolv} implies that the restriction of the projectivization of $\rho$ to a neighborhood of $E$ is solvable, and Proposition \ref{P:naocontrai} implies that the intersection matrix of $E$ is indefinite. Let $D'\subset D$ the union of the components of $D$ around which the monoromy is equal to $\pm\Id$. Clearly $|E|$ is a connected component of $|D-D'|$.
Since $\PSL(\C)$ is a simple group we can apply Theorem \ref{T:criteriofactor} (replacing $D$ by $D-D'$)  to factorize the projectivization of $\rho$ through an orbicurve. Theorem \ref{THM:A} follows. \qed

\section{Riccati foliations}\label{S:basic}

Here we recall basic definitions and properties of Riccati foliations,
 and provide some reduction lemmata. 
 
\subsection{Projective connections and Riccati foliations}

Let $E$ be a rank $2$ vector bundle on a complex manifold $X$ and $\nabla:E\to E\otimes\Omega^1_X(D)$ 
be a meromorphic connection on $E$ with (effective) polar divisor $D$: the operator
$\nabla$ is $\C$-linear and satisfying Leibniz rule 
$$\nabla(f\cdot s)=f\cdot \nabla(s)+df\otimes s\ \ \ \text{for any local section }s\text{ and function }f.$$ 
In any local trivialization $Z:E\to  \C^2$ of the bundle, the connection is defined by 
$$Z\mapsto \nabla(Z)=dZ+AZ\ \ \ \text{with}\ \ \ Z=\begin{pmatrix}z_1\\ z_2\end{pmatrix}\ \text{and}\ 
A=\begin{pmatrix}\alpha&\beta\\ \omega&\delta\end{pmatrix}$$
where $A$ is a matrix of meromorphic $1$-forms (sections of $\Omega^1_X(D)$).

The operator $\nabla$, being linear, commutes with the fiberwise action of $\mathbb G_m=\C^*$ on $E$
and induces a projective connection on the $\P^1$-bundle $\P(E)$. In local trivialization above,
after setting $(1:z)=(z_1:z_2)$, the differential equation $\nabla Z=0$ takes the form $\Omega=0$
for the Riccati $1$-form
$$\Omega=dz+\omega_0+z\omega_1+z^2\omega_2$$
with $(\omega_0,\omega_1,\omega_2)=(\omega,\delta-\alpha,-\beta)$. 
Any two connections $(E,\nabla)$ and $(E',\nabla')$ induce the same projective connection if, and only if
there is a rank $1$ meromorphic connection $(L,\zeta)$ such that $(E',\nabla')=(L,\zeta)\otimes(E,\nabla)$.
The connection $(E,\nabla)$ is actually determined by the (independant) data of the projective connection 
and the trace connection $(\det(E),\mathrm{tr}(\nabla))$ (defined in local trivialization by $d+\alpha+\delta$).

The distribution $\Omega=0$ defines a (singular codimension one) foliation $\mathcal H$ on the total space
$P$ of the $\P^1$-bundle if, and only if it satisfies Frobenius integrability condition $\Omega\wedge d\Omega=0$, 
which is equivalent to
\begin{equation}\label{IntCondOmega}
\left\{\begin{matrix}
d\omega_0=\hfill\omega_0\wedge\omega_1\\
d\omega_1=2\omega_0\wedge\omega_2\\
d\omega_2=\hfill\omega_1\wedge\omega_2
\end{matrix}\right.
\end{equation}
The flatness condition $\nabla\cdot\nabla=0$ for the linear connection, which writes $dA+A\cdot A=0$ in local trivialization,
is equivalent to (\ref{IntCondOmega})  (flatness of the projective connection) and the flatness $d(\alpha+\delta)=0$ for 
the trace connection. In this case, $\nabla$-horizontal sections project onto the leaves of $\mathcal H$.

We say that $(P\to X,\mathcal H)$ is a Riccati foliation since over a general point of $X$, 
the $\P^1$-fiber is tranverse to $\mathcal H$.

\begin{remark}
A codimension one foliation $\mathcal H$ on (the total space of a) $\P^1$-bundle $P\to X$ is a Riccati foliation 
if and only if it is tranverse to a general $\P^1$-fiber.
Indeed, in local trivialization, $\mathcal H$ must be defined by a $1$-form $P(x,z)dz+\tilde{\Omega}(x,z)$ polynomial in $z$.
One easily checks that transversality to $\C=\P^1-\{\infty\}$ implies that $P$ does not depend on the variable $z$. On the other hand, transversality near $z=\infty$ implies that $\tilde{\Omega}$ has degree at most $2$ in $z$.
\end{remark}

\begin{remark} \label{R:generalriccati}One could define define Riccati foliations on arbitrary $\P^1$-bundles, not necessarily of the form $\P(E)$.
However, in this paper, {\it all Riccati foliations are defined on projectivized vector bundles $\P(E)$}. 
In particular, when $X$ is projective, they are birationally equivalent to the trivial bundle $X\times \P^1$ where all above formula
makes sense globally; moreover, they come from the $\sl$-connection on $X\times \C^2$ obtained by setting $\alpha+\delta=0$.
\end{remark}

\begin{remark}
Any two $\sl$-connections $(E,\nabla)$ and $(E',\nabla')$ are equivalent if, and only if,
there exists a flat rank $1$ logarithmic connection $(L,\delta)$ on $X$
having monodromy into the binary group $\{\pm 1\}\subset\mathbb G_m$ such that $(E,\nabla)$ is birationally equivalent to
$(L,\delta)\otimes (E',\nabla')$. In particular, $(E,\nabla)$ and $(E',\nabla')$ are birationally equivalent
after pulling them back to the ramified two-fold cover $Y \to X$ determined by the monodromy representation of  
$(L,\delta)$. 
\end{remark}

\begin{remark}\label{R:curvecase}
If $X$ is  a projective manifold and $H\subset X$ is an hypersurface, there may exist representations
    \[\rho: \pi_1(X-H)\to \PSL(\C)\] which cannot be realized as the monodromy of a Riccati foliation (see \cite[Example 5.2]{RH}). However, when $X$ is a projective curve, then $H=\{p_1,...,p_l\}$ is a finite union of points, this obstruction does not exist because every local monodromy (around each $p_i$) can be realized as the projectivisation of the monodromy of a meromorphic connection. This enables us to construct over $X$ a ${\mathbb P}^1$ bundle (hence a projectivisation of a rank $2$ vector bundle) equipped with a Ricatti foliation with the right monodromy (see  \cite{aln} for the details).
\end{remark}

The (effective) polar divisor $(\mathcal H)_\infty$ of the Riccati foliation is defined 
as the direct image under $\pi:P \to X$ of the tangency
divisor between $\mathcal H$ and the vertical foliation defined by the fibers of $\pi$.
It corresponds to the vertical part of the polar divisor of $\Omega$ in local trivialization. 
We have $(\nabla)_\infty\ge(\mathcal H)_\infty$ since the trace $\mathrm{tr}(\nabla)$ may have
more, or higher order poles; however, in the $\sl$-case, the two divisors coincide.

The {\it monodromy representation} of the Riccati foliation $(P\to X,\mathcal
H)$ is the representation
$$\rho_{\mathcal H}:\pi_1(X \setminus \vert (\mathcal H)_{\infty} \vert)\longrightarrow\PSL(\C)$$
defined by lifting paths on $X\setminus \vert (\mathcal H)_{\infty} \vert$ to the leaves of
$\mathcal H$. We note that $\rho_{\mathcal H}$ is just the projectivization of the linear monodromy of $\nabla$.

\subsection{Riccati foliations defined by a closed $1$-form}\label{sec:RiccClosedForm}
Let us start with two criteria.

\begin{lemma}\label{L:uniquelift}
Let $(P,\mathcal H)$ be a Riccati foliation over a projective manifold $X$.
If there exists a  birational map $\Phi : P \dashrightarrow P$ distinct from the identity  such that $\Phi^* \mathcal H = \mathcal H$
and $\pi \circ \Phi = \pi$ then there exists a generically finite morphism of degree at most two $f:Y\to X$ such that $f^*(P, \mathcal H)$ is defined by a closed rational $1$-form.
\end{lemma}
\begin{proof}
Since $\Phi$ commutes with projection $\pi:P \to X$ it follows that over a general fiber of $\pi$, $\Phi$
is an automorphism. Let $F = \{ z \in P - \indet(\Phi) | \Phi(z) = z \}$ be the set of fixed points of $\Phi$. Since we are dealing with a family of automorphism of $\mathbb P^1$,  the projection of $F$ to $X$ is generically finite of degree one or two. Assume first that the degree is one. Then we can birationally trivialize $P$ in such a way that $F$ becomes the section at infinity and $\Phi$ is of the form $z \mapsto z + \tau$ for some $\tau \in \C(X)$. Let
$
\Omega = dz + \omega_0 + \omega_1 z + \omega_2 z^2
$
be a rational form defining $\mathcal H$. The invariance of $\mathcal H$ under $\Phi$ reads as
\[
\Omega \wedge \Phi^* \Omega =0  \iff \omega_2 = 0 \text{ and } \omega_1 = - d \log \tau  \iff d ( \tau \Omega ) = 0 \, .
\]

If the degree of $\pi\vert_{Z}$ is two then after replacing $X$ by (a resolution of)  a ramified double covering
we can assume that $P$ is trivial and that $\Phi$ is given by $z \mapsto \lambda(x) z$.
The invariance of $\mathcal H$ under $\Phi$ reads as
\[
\Omega \wedge \Phi^* \Omega =0  \iff \omega_0 = \omega_2 = d\lambda = 0   \iff d ( z^{-1} \Omega ) = 0 \, .
\]
This concludes the proof of the lemma.
\end{proof}

\begin{lemma}\label{lem:MultisectionH}
Let $(P,\mathcal H)$ be a Riccati foliation over a complex manifold $X$.
Let $H\subset P$ be a $\mathcal H$-invariant (maybe singular) hypersurface
which intersects the generic fiber of $P\to X$ at $2\le n<\infty$ distinct points.
\begin{itemize}
\item If $n\ge3$, then $\mathcal H$ has  a (non constant) meromorphic first integral.
\item If $n=2$, maybe passing to a two-fold cover $X'\to X$, then $\mathcal H$ is defined
after a convenient bundle trivialization
by $\Omega=\frac{dz}{z}+\omega$ with $\omega$ a closed $1$-form.
\end{itemize}
In particular, the monodromy of $\mathcal H$ is virtually abelian.
\end{lemma}

Since $H$ is $\mathcal H$-invariant, the intersection set between $H$ and a fiber
must be (globally) invariant by the monodromy group computed on the same fiber.
For $n=2$, this implies that the monodromy is dihedral, and for $n>2$, that it is finite.
The conclusion of the Lemma is much stronger in case $\mathcal H$ is irregular,
since $\mathcal H$ could have no monodromy but transcendental leaves in that case.

\begin{proof}
Maybe passing to a finite cover $X'\to X$, we can assume that $H$ splits into $n$ meromorphic sections.
For $n\ge 3$ one can send three of them to $z=0,1,\infty$ and observe that their $\mathcal H$-invariance
implies $\omega_0=\omega_1=\omega_2=0$. Therefore, $z$ is a first integral and all leaves have algebraic closure.
For $n=2$ one can send the two sections to $z=0$ and $z=\infty$ and we get that the Riccati $1$-form
defining $\mathcal H$ takes the form $dz+\omega z$; Frobenius integrability implies $d\omega=0$.
\end{proof}

Let us now describe Riccati foliations defined by closed $1$-forms.

\begin{prop}\label{prop:RiccClosedForm}
Let $(P,\mathcal H)$ be a Riccati foliation over a projective complex manifold $X$.
If $\mathcal H$ is defined by a closed $1$-form $\Omega$ on $P$, then:
\begin{itemize}
\item either $\mathcal H$ has a (non constant) meromorphic first integral,
\item or after a convenient bimeromorphic bundle trivialization, we have
\begin{equation}\label{NormalFormRiccClosedForm}
\Omega=c\left(\frac{dz}{z}+\omega\right) \ \ \ \text{or}\ \ \ \Omega=dz+\omega
\end{equation}
with $\omega$ a meromorphic closed $1$-form on $X$ and $c\in\C^*$.
\end{itemize}
\end{prop}
\begin{proof}
Denote by ${(\Omega)}_0$ and ${(\Omega)}_\infty$ the zero and polar divisor of $\Omega$. Since $\Omega$ is closed, the support $H$ of  ${(\Omega)}_0-{(\Omega)}_\infty$ is $\mathcal H$-invariant.
The hypersurface $H$ intersects the generic fiber of $P\to X$ in $n\ge1$ distinct points
($1$-forms have non trivial divisor on $\mathbb P^1$). If $n\ge3$, then $\mathcal H$ has a first integral.

If $n=2$, then $H$ splits into the union of two meromorphic sections, $H_0$ and $H_\infty$ say; indeed,
either they are zero and polar locus of $\Omega$, or they are both simple pole, but with opposite residue.
After trivialization of $P$ sending them to $z=0$ and $z=\infty$ respectively, we get that $\mathcal H$
is defined by a closed $1$-form $\Omega':=\frac{dz}{z}+\omega$; but $\Omega=f\cdot \Omega'$ is closed
as well, which implies that $f$ is a first integral. We are thus in one of the two items depending on $f$ is constant or not.

Finally, when $n=1$, we can assume $H=\{z=\infty\}$ and we get that $\Omega=f(dz+\omega_0+z\omega_1+z^2\omega_2)$
restrict to a general fiber as $c\cdot dz$, with $c\in\C^*$. This means that $f$ does not depend on $z$
and, by gauge transformation $z\mapsto cz$, we can assume $f=1$. The $1$-form $\Omega$ is closed
if, and only if, $d\omega_0=\omega_1=\omega_2=0$.
\end{proof}

\subsection{Generically finite morphisms and factorizations}

In order to prove Theorems \ref{THM:C} and \ref{THM:D}, it will be useful
to blow-up the manifold $X$ and  pass to a finite cover in order to simplify the foliation.
At the end, we will need the following descent lemma to come back to a conclusion on $X$.

\begin{prop}\label{P:pushRiccati}
Let $(P,\mathcal H)$ be a Riccati foliation over $X$ and assume that $f^*(P,\mathcal H)$ is not defined
by a closed rational $1$-form for any dominant morphism $f : Y \to X$. 
If the pull-back $\varphi^*(P,\mathcal H)$ via generically finite morphism $\varphi:\tilde X\to X$ factors through a curve,
then the same holds true for $(P,\mathcal H)$.
\end{prop}
\begin{proof}It is very similar to the proof of \cite[Lemma 3.6]{Corlette-Simpson}.
Maybe composing by a generically finite morphism, we can assume that
\begin{enumerate}
\item $\varphi$ is Galois in the sense that there is a finite group $G$ of birational transformations
acting on $\tilde X$ and acting transitively on a general fiber;
\item there is a morphism $\tilde f:\tilde X\to\tilde C$ with connected fibers and a Riccati foliation $(\tilde P_0,\tilde{\mathcal H_0})$
over $\tilde C$ such that its pull-back on $\tilde X$ is birationally equivalent to $(\tilde P,\tilde{\mathcal H}):=\varphi^*(P,\mathcal H)$.
\end{enumerate}
In consequence, for any $g\in G$, $g^*(\tilde P,\tilde{\mathcal H})$ is birationally equivalent to $(\tilde P,\tilde{\mathcal H})$
and the Riccati foliation $(\tilde P,\tilde{\mathcal H})$ factors through $\tilde f\circ g: \tilde X\dashrightarrow\tilde C$.
Consider a general fiber $Z$ of $\tilde f$. The Riccati foliation $(\tilde P,\tilde{\mathcal H})\vert_Z$ restricted to $Z$
is birationally equivalent to the trivial Riccati foliation on $Z$, and thus  admits a rational first integral.
If the map $\tilde f\circ g$ were dominant in restriction to $Z$,
this would imply that $(\tilde P_0,\tilde{\mathcal H_0})$ also has a rational  first integral, and the same for $(\tilde P,\tilde{\mathcal H})$
and $(P,\mathcal H)$, contradiction. Thus $g$ (and $G$) must permute general fibers of $\tilde f$ and acts on $\tilde  C$.
Moreover, $g^*(\tilde P_0,\tilde{\mathcal H_0})$ is birational to $(\tilde P_0,\tilde{\mathcal H_0})$ for all $g\in G$. 
Lemma \ref{L:uniquelift} implies that over each $g: \tilde X \dashrightarrow  \tilde X$ there exists a unique birational map $\hat g : \tilde P \dashrightarrow  \tilde P$ such that $\hat g^* \tilde {\mathcal H} = \tilde{\mathcal H}$, and a similar statement holds true for the action of
 $G$ on $\tilde C$. Therefore, we get an action of $G$ on the diagram
$$\xymatrix{
    \tilde P \ar[r] \ar[d] & \tilde P_0 \ar[d] \\
    \tilde X \ar[r]_{\tilde f}  & \tilde C
  }
  $$
which preserves the Riccati foliations. Passing to the quotient, we get a commutative diagram
$$\xymatrix{
    P \ar[r] \ar[d] & P_0 \ar@{.>}[d] \\
    X \ar[r]_f  & C
  }$$
where $P_0\dashrightarrow C$ has $\P^1$ as a general fiber. Moreover, the quotient foliation $\mathcal H_0$
on $P_0$ is transverse to the general fiber and thus of Riccati type.
\end{proof}

\subsection{Monodromy  and factorization}

We say that a Riccati foliation $(P,\mathcal H)$ has regular singularities when it can be locally induced by
a flat meromorphic linear connection having regular singularities in the sense of \cite[Chapter II]{Deligne}.
Like in the linear case, any two  Riccati foliations $(P,\mathcal H)$ and $(P',\mathcal H')$ having  regular singularities have the same monodromy if and only if there exists a birational bundle map $\phi: P \dashrightarrow P'$ such that $\phi^*\mathcal H' = \mathcal H$, see \cite[Lemma 2.13]{GaelJorge}.
In particular, if the monodromy factors through a curve, then so does the Riccati foliation by remark \ref{R:curvecase}.
The next proposition, borrowed from \cite[Proposition 2.14]{GaelJorge}, tells us what remains true in the irregular case.

\begin{prop}\label{P:pullRiccati}
Let $(P,\mathcal H)$ be a Riccati foliation over $X$. Suppose there exists a morphism $f: X \to C$ with connected fibers such that the polar divisor  $(\mathcal H)_{\infty}$  of $\mathcal H$  intersects the general fiber of $f$ at most on regular singularities;
assume moreover that
the monodromy representation $\rho$ of $(P,\mathcal H)$ factors through $f$, i.e.  there exists a divisor $F$ supported on finitely many  fibers of $f$ and a representation $\rho_0$ from the fundamental group of $C_0= f(X -  |(\mathcal H)_{\infty} + F|)$ to $\PSL(\C)$ fitting in the diagram below.
\begin{center}
\begin{tikzpicture}
  \matrix (m) [matrix of math nodes,row sep=2em,column sep=4em,minimum width=2em]
  {
    \pi_1(X - |(\mathcal H)_{\infty} + F| ) & \PSL(\C)  \\
      \pi_1(C_0 ) & \\};
  \path[-stealth]
    (m-1-1) edge node [above] {$\rho$} (m-1-2)
           edge node [left] {$f_*$} (m-2-1)
    (m-2-1)  edge node [below] {$\rho_0$} (m-1-2)  ;
\end{tikzpicture}
\end{center}
Then $(P,\mathcal H)$ factors through $f: X \to C$.
\end{prop}

\begin{remark}For a Riccati or a linear connection, to have regular singularities  is a local property which,
if satisfied at some point of the polar divisor, remains true all along the irreducible component.
The support of the polar divisor splits into regular and irregular components.
The assumption in Proposition \ref{P:pullRiccati} that the general fiber of $f$ intersect only regular singularities
just means that the fiber intersects only regular components of the polar locus. Equivalently,
the connection (or Riccati foliation) restricts to the fiber as a regular singular connection (resp. Riccati foliation).
\end{remark}

\subsection{Reduction to the two-dimensional case}\label{sec:ReductionDim2}

The proposition below
allows us to reduce our study of Riccati foliations over arbitrary projective manifolds
to study of Riccati foliation over projective surfaces.

\begin{prop}\label{P:reduction2}
Let $(P,\mathcal H)$ be a Riccati foliation over a projective manifold $X$ and assume it has no rational first integral. If the restriction of
$(P,\mathcal H)$ to a sufficiently  general surface $S \subset X$ factors through a curve, then the same holds true
for $(P,\mathcal H)$ over $X$.
\end{prop}
\begin{proof}
Denote by $\pi: P \to X$ the natural projection.
Let $(P_S, \mathcal H_S)$ be the restriction of $(P,\mathcal H)$  to $S \subset X$, i.e.
$P_S = P\vert_{\pi^{-1}(S)}$ and $\mathcal H_S = \mathcal H\vert_{\pi^{-1}(S)}$.
By assumption, we get a rational bundle map $\phi : P_S \dashrightarrow P_0$ such that $\phi^* \mathcal H_0 =  \mathcal H_S$.
This shows that $\mathcal H_S$  contains a foliation by algebraic curves.

Applying the same argument for different choices of $S$, we obtain that
through a general point $p \in P$ the leaf of $\mathcal H$ through $p$ contains an algebraic subvariety $A_p$.
It is known that the germ of leaf at a general point contains an unique germ of algebraic subvariety
maximal with respect to inclusion which turns out to be irreducible, see  the proof of \cite[Lemma 2.4]{Croco3}.
In our setup, we can thus assume that $A_p$ has codimension at most two, since otherwise we would be able to enlarge $A_p$ by  choosing an appropriate $S$ and applying the above argument.

It follows from \cite[Lemma 2.4]{Croco3}
that $\mathcal H$ contains a foliation $\mathcal G$ with all leaves algebraic and having
codimension at most  two, and that $\mathcal H$ is the pull-back of a foliation on a curve or a surface.
In the former case, we get a rational first integral for $\mathcal H$, contradiction.
Hence the codimension of $\mathcal G$ must be equal to two. Let now $L$ be a general leaf of $\mathcal G$ and
consider its projection to $X$. Since $\mathcal H$ is tranverse to the general of fiber of $\pi$, this
projection is generically \'{e}tale and $\pi(L)$ is (a Zariski open subset) of a divisor $D_L$ on $X$.
The construction method of $\mathcal G$ makes clear that  $\pi^{-1}(D_L)$ is invariant by $\mathcal G$,
and also that the restriction of $\mathcal H$ to $\pi^{-1}(D_L)$ is birationally equivalent to
the foliation on a trivial $\P^1$-bundle over $D_L$ given by the natural projection  to $\P^1$.
This is sufficient to show that $\mathcal H$ is the pull-back of a Riccati foliation on a $\P^1$-bundle
over a curve.
\end{proof}

\begin{prop}\label{P:reduction2closed}
Let $(P,\mathcal H)$ be a Riccati foliation over a projective manifold $X$, $\dim(X)\geq 2$.
If the restriction of
$(P,\mathcal H)$ to a  general hyperplane section $Y\subset X$ 
is defined by a closed rational $1$-form, then the same holds true for $(P,\mathcal H)$ over $X$.
\end{prop}

\begin{proof}Since being defined by a closed rational $1$-form is invariant under birational transformation,
we can assume without loss of generality that $P=X\times\mathbb P^1$. 
Set $Z=Y\times\mathbb P^1\subset P$. 
If $Y$ is sufficiently general, then $Z$ is transversal to $\mathcal H$ outside a proper algebraic subset of $\mathcal H$ and we can consider $\omega$,  the closed (non zero) rational $1$-form defining the restriction $\mathcal H\vert_Z$ on $Z$.
Then, following \cite[p.47-50]{CerveauMattei} (see also \cite[\S 5.6]{CanoCerveauDeserti}),
$\omega$ extends (uniquely) as closed a meromorphic $1$-form $\Omega$ defining $\mathcal H$ on the neighborghood 
of $Z$. Let $U$ some connected neighborhood of $Y$ over which $\Omega$ is defined and then can be written as \[\Omega=a_0 dz +a_1\omega_1+...+a_n\omega_n\] where $\omega_1,...,\omega_n$ are $n=\mbox{dim}_{\mathbb C} X$ meromorphically independent $1$-forms on $X$  and the $a_i$'s are rational functions of $z$ (a projective coordinate on ${\mathbb P}^1$) with coefficients  lying in the field of meromorphic functions on $U$. The extension of those coefficients on the whole $X$  is then automatic and implies the extension of $\Omega$ through $X\times {\mathbb P}^1$.
\end{proof}

\section{Polar divisor and reduction of Riccati foliations}\label{S:polar}
In the next two sections, we will restrict our attention to Riccati foliations over projective surfaces.
Proposition \ref{P:reduction2} allows to transfer the conclusions to Riccati foliations over arbitrary projective manifolds.
The general case  will come back into play  only at the proofs of Theorems \ref{THM:C} and \ref{THM:D}.

Let  $(P,\mathcal H)$ be a Riccati foliation over a projective surface $S$.
In this section, we review the local structure of $(P,\mathcal H)$ over a neighborhood 
of a general point of $(\mathcal H)_\infty$. Maybe blowing-up $S$ and passing to a ramified
cover, we arrive to a list of nice local models that either are defined by a first integral,
or factor through a curve (a local version of Theorem \ref{THM:C}).
Let us first recall the one dimensional classification following the description
of \cite[Chapter 4, Section 1]{Brunella}.

\subsection{A review of the one dimensional case}\label{SecReview1D}

Here, we follow the description of \cite[Chapter 4, Section 1]{Brunella} which 
corresponds to Levelt-Turritin normalization.

\begin{prop}\label{prop:BrunellaTurritin}
Let $(P\to \Delta,\mathcal H)$ be a Riccati foliation over a disc $0\in\Delta\subset\C$.
Maybe after passing to the two-fold cover 
$$\Delta\to\Delta\ ;\ x\mapsto x^2,$$
and after bimeromorphic bundle trivialization and change of $x$ coordinate, 
the Riccati $1$-form $\Omega$ defining the foliation fits into one of the following types:
\begin{itemize}
\item  $\Omega=dz$ (trivial case)
\item {\bf Log${}^{\mathbb G_m}$}: $\frac{\Omega}{z}=\frac{dz}{z}+\lambda\frac{dx}{x}$, with $\lambda\in\C\setminus\Z$;
\item {\bf Log${}^{\mathbb G_a}$}: $\Omega=dz+\frac{dx}{x}$;
\item {\bf Irreg}: $\Omega=dz+\left(\frac{dx}{x^{k+1}}+\lambda\frac{dx}{x}\right)z+(a(x)z^2+b(x)z+c(x))dx$ with $a,b,c$ holomorphic,
$k\in\Z_{\ge 1}$ and $\lambda\in\C$.
\end{itemize}
\end{prop}

Since the statement does not appear in this form in the litterature, 
we provide a sketch of proof as well as definitions and useful remarks.
 
Let us start with a Riccati foliation $\mathcal H$ defined near $x=0$ in $\mathbb C_x\times\mathbb P^1_z$ by 
$$dz+\left(a(x)z^2+b(x)z+c(x)\right)dx=0.$$
Here, $a,b,c$ are meromorphic at $x=0$. If $a,b,c$ are holomorphic, then 
the Riccati foliation is equivalent to $dz=0$ by a biholomorphic bundle tranformation
(choose $3$ distinct solutions and send them to $z=0,1,\infty$).
If $a,b,c$ have at most simple poles, then by a bimeromorphic bundle tranformation
we can reduce $\mathcal H$ to one of the following models (see \cite[loc.cit]{Brunella})
$$dz+\lambda\frac{dx}{x}z=0,\ \ \ dz+\frac{dx}{x}=0\ \ \ \text{or}\ \ \ dz=0$$
with $\lambda\in\mathbb C\setminus\mathbb Z$ (for $\lambda\in\mathbb Z$, 
we get $d\tilde z=0$ for $\tilde z=x^\lambda z$). Note that the trivial case can really occur even if the original Riccati admits a pole at $x=0$ (see  \cite[loc.cit]{Brunella}). So far, the coordinate $x$ 
remains unchanged. The monodromy of $\mathcal H$ around $x=0$ is respectively 
given by $z\mapsto e^{2i\pi\lambda}z$, $z\mapsto z+1$ or trivial.
This normal form is unique except the choice of $\lambda$: 
bundle transformations $z\mapsto x^nz^{\pm1}$ 
change $\lambda$ into $\pm\lambda+n$; the birational invariant is $\cos(2\pi\lambda)$.

From now on, let us assume that the Riccati equation has a multiple pole at $x=0$, even up to 
bimeromorphic bundle transformations. We are in the so-called {\bf irregular case}.

We may assume $z=\infty$ is not $\mathcal H$-invariant,
i.e. $a\not\equiv0$. Then, after a (unique) bimeromorphic bundle transformation of the form $z\mapsto f(x)z+g(x)$,
with $f,g$ meromorphic at $x=0$, $f\not\equiv0$, we may assume the Riccati equation in the form
\begin{equation}\label{eq:companion}
dz+\left(z^2+\frac{\phi(x)}{x^{n+2}}\right)dx=0.
\end{equation}
(the corresponding $\mathfrak{sl}_2$-matrix connection is in companion form).
Here, we assume $\phi$ holomorphic non vanishing and $n\in\mathbb Z_{\ge0}$ 
is the order of pole. Define the {\bf irregularity index} (or Poincar\'e-Katz rank)
to be 
$$\kappa:=\frac{n}{2}\in\frac{1}{2}\mathbb Z_{\ge0}$$
(see \cite{Andre}). By an additional bundle transformation $z\mapsto x^{k+1}z$
with $k=\kappa$ or $\kappa+\frac{1}{2}$
(depending on the parity of $2\kappa$) we get the form
$$dz+\frac{z^2+(k+1)x^kz+\phi(x)}{x^{k+1}}dx=0.$$
Then we cannot decrease the order of poles by further bimeromorphic bundle transformations
(here we actually assume $\kappa>0$ otherwise we are back to the logarithmic case).

We say that we are in the {\bf unramified case} if $\kappa\in\mathbb Z_{>0}$, which is equivalent 
to say that $\phi(0)\not=0$. In this case, the Riccati foliation has two singular points, 
and through each of them passes a (unique) formal $\mathcal H$-invariant section,
i.e. a solution $z=f(x)$ of the Riccati equation with $f(x)$ a possibly divergent power series.
After sending these two formal leaves to $z=0$ and $z=\infty$ by a formal bundle transformation,
the Riccati equation becomes
$$\frac{dz}{z}+\omega=0\ \ \ \text{with}\ \ \ \omega=\frac{\tilde\phi(x)}{x^{k+1}}dx$$
with $\tilde\phi(x)$ formal power series. Usually, we kill the holomorphic part of $\omega$
by a last (formal) bundle modification of the form $z\mapsto h(x)z$, and the principal part 
of $\omega$ is an invariant by birational bundle transformations. Instead of this, 
we can use a (formal) change of $x$-coordinate to put the $1$-form into the normal form 
$\omega=\frac{dx}{x^{k+1}}+\lambda\frac{dx}{x}$ (see \cite[proof of Proposition 1.1.3]{Frankpseudo}).
In fact, a combination of both operations shows that, after analytic change of $x$-coordinate
and holomorphic bundle transformation, we can put the Riccati equation into the form
{\bf Irreg} of Proposition \ref{prop:BrunellaTurritin}: we use analytic approximation of 
formal solutions to get $a,c$ holomorphic and then change $x$-coordinate up to the order $k$
to normalize the principal part of $\omega$. The formal normal form (see also \cite[Section IV.1.2]{MartinetRamis1})
\begin{equation}\label{eq:FormalNormFormUnramified}
\frac{dz}{z}+\frac{dx}{x^{k+1}}+\lambda\frac{dx}{x}.
\end{equation}
can be obtained by another formal bundle transformation.
As in the logarithmic case, the residue $\lambda$ is not unique:
bundle transformations $z\mapsto x^nz^{\pm1}$ 
change $\lambda$ into $\pm\lambda+n$.

In the {\bf ramified case} (called nilpotent in \cite{Brunella}) $\kappa=k-\frac{1}{2}$ we get, 
after bimeromorphic bundle transformation, the following normal form 
\begin{itemize}
\item {\bf Irreg${}^{\text{ram}}$}: $\Omega=dz+\frac{z^2-2x^kz-x}{4x^{k+1}}dx+(a(x)z^2+b(x)z+c(x))dx$ with $a,b,c$ holomorphic and $k\in\Z_{\ge 1}$.
\end{itemize}
After ramification $x=\tilde x^2$, we get a singularity of irregular (unramified) type {\bf Irreg} with irregularity
index $\tilde k=2\kappa$ (and $\lambda=0$). More generally, all three types {\bf Log${}^{\mathbb G_m}$}, 
{\bf Log${}^{\mathbb G_a}$} and {\bf Irreg} are stable under ramifications $x=\tilde x^n$;
moreover $\lambda$ and the irregularity index $\kappa$ are both multiplied by $n$.

\begin{remark}\label{RemAlgoBrunella}
As explained in \cite[Chapter 4, Section 1]{Brunella}, there is a more geometric algorithmic way 
to minimize the order of pole of a Riccati foliation $(\pi:P\to\Delta,\mathcal H)$ at $0\in\Delta$.
Recall that an elementary transformation of $P$ at a point $p\in P_0:=\pi^{-1}(0)$ consists in blowing-up 
the point $p\in P$ and then contracting the strict transform of the fiber $P_0$. 
Assume $\mathcal H$ has a multiple pole at $0\in\Delta$. 
\begin{itemize}
\item If $\mathcal H$ has $2$ singular points on the fiber $P_0$, then the order of poles is minimal,
{\it i.e.} cannot be decreased by bimeromorphic bundle transformation: we are in the 
irregular unramified case. 
\item If $\mathcal H$ has only $1$ singular point $p$ on the fiber $P_0$, then apply an elementary transformation at $p$; if moreover the order of pole has not decreased, then again it is minimal:
we are in the irregular ramified case. 
\end{itemize}
Finally, when the order of pole is not minimal,
and we want to minimize it, we just have to apply an elementary transformation to the 
unique singular point of $\mathcal H$, and possibly repeat this operation, until we arrive 
at a simple pole or at one of the aforementioned two cases. This algorithmic procedure can easily be 
generalized for Riccati foliations $(P\to S,\mathcal H)$ over a surface $S$ (see \cite{RH}).
\end{remark}

\subsection{Irregular singular points and Stokes matrices}\label{S:stokes}
For details on what follows, see \cite[Chapitre VI]{MartinetRamis1} or \cite[Section 5]{Frankpseudo}.
In the normal form {\bf Irreg}, coefficients $a,b,c$ can be killed by formal (generally divergent) gauge transformation and
we arrive at the formal normal form (\ref{eq:FormalNormFormUnramified}):
$$\frac{\Omega}{z}=\frac{dz}{z}+\frac{dx}{x^{k+1}}+\lambda\frac{dx}{x}.$$
In the setting of $\mathfrak{sl}_2$-connections, this normal form writes
$$d+\begin{pmatrix} -\frac{\alpha }{2} & 0\\ 0 &\frac{\alpha}{2} \end{pmatrix}$$
where $\alpha= \frac{dx}{x^{k+1}}+\lambda\frac{dx}{x}.$

In general this last normalization is not convergent. Nevertheless, there are $2k$ closed
sectors covering a neighborhood of $0$ in the $x$-variable such that the differential equation is holomorphically conjugate
to the normal form over the interior of the sector, and the conjugation extends continuously to the boundary. Each of
the sectors contains exactly one of the arcs $\{ x \in \mathbb D_{\varepsilon}   | x^k \in i \mathbb R \} - \{ 0\}$ where $\mathbb D_{\varepsilon}=\{|x|<\varepsilon\}$.
Over each of these  sectors there are only two solutions with well defined limit when $x$ approaches zero which correspond to the solutions  $\{z=0\}$ and $\{z= \infty\}$ for the normal form.
When we change from a sector intersecting $x^k \in i\mathbb R_{>0}$
to a sector intersecting $x^k \in i\mathbb R_{<0}$ in the counter clockwise direction
then we can continuously extend the solution corresponding to $\{ z= + \infty \}$
but the same does not hold true  for the solution corresponding to $\{ z = 0\}$.
Similarly, when changing from sectors intersecting $x^k \in i\mathbb R_{<0}$ to
 sectors with points in $x^k \in i\mathbb R_{>0}$ in the counter clockwise direction we can extend continuously the solution corresponding to   $\{ z= 0\}$  but not
 the one corresponding to $\{ z = \infty\}$.

\begin{figure}
\begin{center}
\begin{tikzpicture}[scale=0.9,font=\small]
    \begin{scope}
     \draw (150:3) node {$Re(x^3)=0$}    ;
      \draw (0,-2.7) node {Sectors containing $x^3 \in i \mathbb R_{>0}$}    ;
    \foreach \angle in {30,90,150,210,270,330}
      \draw[dashed] (0,0) -- (\angle :2.3);

    \foreach \angle in {30,150,270 }
     \path [draw ,fill= lightgray,semitransparent] (0,0) -- (\angle-40:1.5) arc (\angle-40:\angle + 40 :1.5) -- (0,0);

    \foreach \angle in {30, 150, 270 }
     \path [draw ,thick] (0,0) -- (\angle-40:1.5) arc (\angle-40:\angle + 40 :1.5) -- (0,0);

    \end{scope}

      \begin{scope}[shift={(6,0)}]
           \draw (330:3) node {$Re(x^3)=0$}    ;
           \draw (0,-2.7) node {Sectors containing $x^3 \in i \mathbb R_{<0}$}    ;
        \foreach \angle in {30,90,150,210,270,330}
      \draw[dashed] (0,0) -- (\angle :2.3);

    \foreach \angle in {90,210,330}
     \path [draw ,fill= lightgray,semitransparent] (0,0) -- (\angle-40:1.5) arc (\angle-40:\angle + 40 :1.5) -- (0,0);

    \foreach \angle in {90,210,330}
     \path [draw ,thick] (0,0) -- (\angle-40:1.5) arc (\angle-40:\angle + 40 :1.5) -- (0,0);

    \end{scope}
\end{tikzpicture}
\end{center}
\caption{}
\end{figure}

The obstructions to glue continuously the two distinguished  solutions over adjacent sectors are the only obstructions to
analytically conjugate the differential equation to its normal form.   These obstructions are codified by the  Stokes
matrices (matrix point of view is more convenient here):
$$\begin{pmatrix}1&b_1\\0&1\end{pmatrix}\begin{pmatrix}1&0\\c_1&1\end{pmatrix}\ \cdots\ \begin{pmatrix}1&b_k\\0&1\end{pmatrix}\begin{pmatrix}1&0\\c_k&1\end{pmatrix}$$
(well-defined up to simultaneous conjugacy by a diagonal matrix). 
Precisely, the $b_i$'s (resp. the $c_i$'s) are responsible for the divergence of the central manifold at $z=\infty$ (resp. $z=0$).
In other terms, the $b_i$'s (resp. the $c_i$'s) are the obstructions to kill the coefficient $a(x)$ (resp. $c(x)$). The monodromy around $x=0$ is given by multiplying this sequence of Stokes matrices (in this cyclic order) with the formal monodromy $$\begin{pmatrix}e^{-i\pi\lambda}&0\\0&e^{i\pi\lambda}\end{pmatrix}$$
(on the left or the right, as this does not matter up to diagonal conjugacy).

If  $x=\varphi(\tilde x)$ is a change of coordinate compatible with normal form {\bf Irreg}, then 
the linear part $\varphi'(0)$ is a $k^{\text{th}}$ root of unity; it permutes the sectors, and consequently 
induces a cyclic permutation of (indices of) Stokes matrices.
All these classical results can be found in \cite[Chapitre VI.2]{MartinetRamis1}.
Furthermore, one also finds there explicit examples of Riccati foliations with non trivial Stokes matrices.
The most famous of them
is undoubtly  Euler's equation that can be interpreted as a  Riccati foliation over $\mathbb P^1$ defined by the
rational $1$-form
\[
dz - \frac{(z- x)}{x^2} dx.
\]
The pole over $\{x=0\}$ is   irregular unramified with non-trivial Stokes, since the weak separatrix through $(0,0)$
 is divergent.

\subsection{Local factorization at a generic point of the polar divisor}\label{sec:locfac}
After reviewing the situation in dimension one we now move back to the two dimensional case.
Let $ (\pi: P \to S,\mathcal H)$ be a Riccati foliation
over a surface $S$ and let $H\subset S$ be an irreducible component of the polar divisor $(\mathcal H)_\infty$.
The hypersurface $\pi^{-1}(H)$ can be $\mathcal H$-invariant or not. In fact, the following assertions are equivalent (see \cite[Lemma 4.1]{RH})
\begin{itemize}
\item $H$ is not $\mathcal H$-invariant,
\item the polar order of $d\Omega$ along $\pi^{-1}(H)$ is strictly greater than the polar order of $\Omega$.
\end{itemize}
Moreover, in this case, if $H$ is smooth (e.g. if $(\mathcal H)_\infty$ is simple normal crossing), then we can make a bundle modification over $H$ such that 
$H$ is no more a polar component of $\mathcal H$. 

The singular set of $\mathcal H$ is located over the polar divisor,
and in restriction to $\pi^{-1}(H)$, it consists in a curve $\Gamma$ intersecting the generic fiber $\pi^{-1}(p)$, 
$p\in H$
in $1$ or $2$ points. Vertical irreducible components $\pi^{-1}(p)$ of $\Gamma$ occur over singular points 
of $(\mathcal H)_\infty$ and so-called {\it turning points} for the connection. Let us call them {\it special points}.
The remaining (non vertical) part of $\Gamma$ can consist in one or two sections, or also of an irreducible double section.
When $H$ is $\mathcal H$-invariant, then we have (see \cite[Lemma 4.1]{RH})

\begin{lemma}\label{L:trivializa}
Let  $H$ be an irreducible component of the polar divisor of $(P,\mathcal H)$.
If $\pi^{-1}(H)$ is $\mathcal H$-invariant then the foliation $\mathcal H$
locally factors through a curve along $H$ minus its set of special points.
\end{lemma}

For the sake of completeness, we will give a short proof in Section \ref{SecBasicLemma}.
By local factorization we mean the following. At any non special point $p\in H$,
and for a sufficiently small disk $\Delta\subset S$ tranverse to $H$ at $p$,
there exists a neighborhood $U\subset S$ of $p$ and 
a submersion $f:U\to\Delta$ with $f^{-1}(p)=H\cap U$
such that the Riccati $(P,\mathcal H)\vert_U$ over the neighborhood $U$ is
 the pull-back of its restriction over $\Delta$  through a  fibre bundle isomorphism.
In other words, the Riccati foliation is locally a product
of a Riccati  foliation over a disk  by a disk (or a polydisk in higher dimension).
In this situation, the isomorphism class of the Riccati foliation $(P,\mathcal H)\vert_\Delta$ is called
{\it the transverse type} of $(P,\mathcal H)$ along the component $H$.

\begin{remark}For a linear $\mathfrak{sl}_2$-connection $(E,\nabla)$, the corresponding Riccati foliation
obtained by projectivization satisfies the assumption of Lemma \ref{L:trivializa} at a non special point $p\in H$
of the polar divisor if and only if the polar divisors of the matrix connection $A$ of $\nabla$ and its differential $dA$ satisfies in any local trivialization of $E$,
 $(dA)_\infty\le (A)_\infty$. In the case
of simple poles, this just means that the pole is actually logarithmic. 
\end{remark}

The following result follows from the existence of a canonical lattice (see \cite{Malgrange})
and is also proved in \cite{RH} for the Riccati setting. 

\begin{prop}Let $ (\pi: P \to S,\mathcal H)$ be a Riccati foliation
over a projective surface $S$. Up to birational bundle modification, we can assume that 
all irreducible components $H$ of $(\mathcal H)_\infty$ satisfy assumption of Lemma \ref{L:trivializa}
and the Riccati foliation locally factors through a curve at the neighborhood of any non special point $p\in S$.
\end{prop}

We now explain how special points can be simplified by blowing-up and finite flat morphism.

\subsection{Reduction of singularities for Riccati foliations over surfaces}\label{S:goodformal}

The following result will be proved in  Appendix \ref{S:normalform}.

\begin{thm}\label{THM:ReductionRiccati}
Let $ (\pi: P \to S,\mathcal H)$ be a Riccati foliation over a projective surface $S$.
There exists a generically finite morphism $\phi:\tilde S\to S$ (with $\tilde S$ smooth) 
and a birational bundle modification of the pull-back bundle $\phi^*P$ such that
the pull-back Riccati foliation has normal crossing polar divisor $D$, and is locally defined
over any point $p\in \vert D\vert$ by a Riccati $1$-form having one of the following types:
\begin{enumerate}
\item\label{ModelLogGm} {\bf Log${}^{\mathbb G_m}$}: $\frac{\Omega}{z}=\frac{dz}{z}+\lambda_x\frac{dx}{x}$, or $\frac{\Omega}{z}=\frac{dz}{z}+\lambda_x\frac{dx}{x}+\lambda_y\frac{dy}{y}$;
\item\label{ModelLogGa} {\bf Log${}^{\mathbb G_a}$}: $\Omega=dz+\frac{dx}{x}$, or 
$\Omega=dz+\frac{dx}{x}+\lambda\frac{dy}{y}$;
\item\label{ModelIrreg0} {\bf Irreg${}^{0}$}: $\frac{\Omega}{z}=\frac{dz}{z}+\frac{df}{f^{k+1}}+\lambda_x\frac{dx}{x}+\lambda_y\frac{dy}{y}$,
where $f=x$ or $xy$;
\item\label{ModelIrreg} {\bf Irreg}:  $\Omega=dz+\left(\frac{df}{f^{k+1}}+\lambda\frac{df}{f}\right)z+\left(a(f)z^2+b(f)z+c(f)\right) df$, where $f=x$ or $xy$.
\end{enumerate}
In all cases, we have $\lambda\in\C^*$, $\lambda_x,\lambda_y\in\C\setminus\mathbb Z$, $a,b,c$ holomorphic and $k\in\Z_{\ge 1}$.
\end{thm}

The first part of the proof has to be compared with Sabbah's result in \cite{Sabbah}. 
He defines, for linear meromorphic connections, 
a notion of good formal model at a point of a normal crossing divisor; this  allows him to define Stokes matrices.
Bad formal models occur in codimension $2$; they correspond to special points of Section \ref{sec:locfac}.
When the base manifold $X$ is a surface (and up to rank $5$ connection), Sabbah proved that, 
maybe blowing-up $X$, we can assume that all points of the polar divisor have good formal model; 
this was generalized by Kedlaya and Mochizuki for any rank \cite{Kedlaya,Mochizuki}. 

In Appendix \ref{S:normalform}, for the surface and rank $2$ case, we provide a proof using an auxiliary (transversely projective) foliation, 
Seidenberg's resolution of singular points and the classification of reduced singular points of transversely projective
foliations by Berthier and the third author in \cite{BerthierTouzet, FredTransvProj}. We use ramified coverings to get rid of ramified irregular singular points.

In the sequel, a Riccati foliation $(\pi: P \to S,\mathcal H)$ will be said in {\bf reduced} 
if it satisfies the conclusion of Theorem \ref{THM:ReductionRiccati}.

\subsection{Closed $1$-form and Stokes matrices}\label{sec:ClosedFormStokes}
Let $(P,\mathcal H)$ be a reduced Riccati foliation on a complex surface $S$.
Let $p\in S$ be a point of the support of the divisor $(\mathcal H)_\infty$.
If the local model for $\mathcal H$ at $p$ is of type (\ref{ModelLogGm}),
(\ref{ModelLogGa}) or (\ref{ModelIrreg0}) in Theorem \ref{THM:ReductionRiccati},
we see that $\mathcal H$ is locally defined by a meromorphic closed $1$-form. 
In the last item (\ref{ModelIrreg}) of Theorem \ref{THM:ReductionRiccati},
the coefficients $a(f),b(f),c(f)$ can be killed by a formal bundle transformation $\hat z=\hat\phi(z)$,
i.e. with $\hat\phi\in\mathrm{PGL}_2(\C[[f]])$. By this way, we arrive at the normal form
(see \S \ref{SecReview1D} or \cite[Section IV.1.2]{MartinetRamis1})
$$\hat{\Omega}=\frac{d\hat z}{\hat z} + \left(\frac{\lambda}{f} + \frac{1}{f^{k+1}}\right) df$$
which is closed. From this formal model, one easily check that $\hat z=0$ and $\hat z=\infty$
are the only formal sections on the bundle that are $\mathcal H$-invariant (i.e. in restriction to which
$\hat{\Omega}$ identically vanish). Therefore, $\Omega$ has  exactly two $\mathcal H$-invariant formal sections of $P$ at $p$. 
The obstruction to the convergence of these two formal sections is given by
non diagonal coefficients of Stokes matrices (see \cite[Chap.III, Th. 4.4]{MartinetRamis1}).

\begin{lemma}\label{lem:ClosedFormStokes}
At a sufficiently small neighborhood $V_p$ of $p$, the following assertions
are equivalent:
\begin{itemize}
\item $\mathcal H\vert_{V_p}$ is defined by a meromorphic closed $1$-form on the restriction $P\vert_{V_p}$,
\item there are two analytic sections $V_p\to P$ of the bundle that are $\mathcal H$-invariant,
\item Stokes matrices are all trivial at $p$ if the type is (\ref{ModelIrreg}).
\end{itemize}
In this case, the two analytic $\mathcal H$-invariant sections induce the formal sections at $p$.
\end{lemma}
\begin{proof}If $\mathcal H\vert_{V_p}$ is defined by a meromorphic closed $1$-form $\Omega$,
then it follows from Proposition \ref{prop:RiccClosedForm} that $\Omega=c\cdot \left(\frac{dz}{z}+\omega\right)$
after convenient bundle trivialization,
with $\omega$ a meromorphic closed $1$-form on $X$ and $c\in\C^*$.
Indeed, $\mathcal H$ cannot have a meromorphic first integral over $p$ since it is irregular,
and it cannot be defined by $dz+\omega$ since it has two $\mathcal H$-invariant formal sections.
Therefore, $z=0$ and $z=\infty$ are two $\mathcal H$-invariant sections which must coincide
with the two formal ones at $p$. Conversely, if $\mathcal H$ has two analytic sections,
then it is defined by a closed $1$-form (see Lemma \ref{lem:MultisectionH} and its proof).
Finally, that the two last assertions are equivalent is well-known (see \cite[Chap.III, Th. 4.4]{MartinetRamis1}).
\end{proof}

\begin{remark} At the neighborhood of an irregular singular point $p$,
a Riccati foliation $\mathcal H$ has no other local (formal) multi-section than the two formal sections discussed above.
It has no meromorphic (formal) first integral and the unique formal $1$-form defining $\mathcal H$
is, up to a scalar constant, the $1$-form $\hat{\Omega}$ above.
\end{remark}

\section{Irregular divisor}\label{S:irregular}

Throughout this section, $(P,\mathcal H)$ is a {\bf reduced} Riccati foliation over a projective surface $S$,
i.e. $\mathcal H$ is as in  the  conclusion of Theorem \ref{THM:ReductionRiccati}. 
We study the irregular part of the polar divisor of $\mathcal H$. Precisely, since $\mathcal H$ is reduced, 
we have $(\mathcal H)_\infty=\sum_i (1+k_i)D_i$ where $k_i$ is the irregularity index of $\mathcal H$ along $D_i$;
in particular, $k_i=0$ precisely if $\mathcal H$ is logarithmic at the generic point of $D_i$. Then we can decompose
$$(\mathcal H)_\infty=(\mathcal H)_{\infty,red}+I$$
where $(\mathcal H)_{\infty,red}$ is a reduced divisor with support $\vert (\mathcal H)_\infty \vert$, 
and $I:=\sum_i k_iD_i$ denotes the {\bf irregular divisor}.

We note that the irregularity index $k_i$ must be constant along each connected component of $I$, due to the local 
model of $\mathcal H$ at intersection points $p\in D_i\cap D_j$ (see (\ref{ModelIrreg0}) and (\ref{ModelIrreg}) of Theorem \ref{THM:ReductionRiccati}).

We also note that the singular set $\sing(\mathcal H)$ over the irregular divisor $I$
consists of finitely many $\P^1$-fibers, located over intersection points of the polar divisor,
and a smooth curve $Z\subset\pi^{-1}(I)$ on which $\pi$ induces a $2$-fold \'etale covering
$\pi\vert_Z:Z\to I$. If we restrict this picture to a connected component $D$ of $I$,
the curve $Z\vert_D$ may split as a union of two sections, say $Z_0$ and $Z_\infty$, 
or be irreducible (i.e. unsplit). Throughout the section, we will refer to this dichotomy
as the {\bf split} or {\bf unsplit case}. Note that, if we are working on a small analytic neighborhood $V$ of $D$, with $V$ and $D$ having the same homotopy type, then
maybe passing to a $2$-fold \'etale covering $\tilde V\to V$, we can always assume that we are in the split case. Indeed, the monodromy of $\pi\vert_{Z}$ induces a representation 
$\pi_1(V)\simeq\pi_1(|D|) \to \mathbb Z/ 2 \mathbb Z$ which defines such a covering.

\subsection{Flat coordinates}\label{Sec:Flat} Here, we show that any connected component 
$D$ of the irregular divisor $I$ has torsion normal bundle, in particular $D\cdot D=0$.
Indeed something even stronger holds true as proved in the proposition below. 

\begin{prop}\label{P:flat}
Let $(P,\mathcal H)$ be a reduced Riccati foliation on a projective surface $S$.
Let $D= k \cdot D_{red}$ be a connected component of the irregular divisor $I$.  
Then the normal bundle of $D_{red}$ is torsion of order $r$ dividing $2k$
(dividing $k$ in the split case) and moreover
$$\mathcal O_S(r D_{red})\vert_{k D_{red}}\simeq \mathcal O_{k D_{red}}.$$
In particular, in any case, we have $D\cdot D=0$.
\end{prop}

\begin{proof}
Take an open covering of a neighborhood of $\vert D\vert$ by sufficiently small open sets $V_i$ on which $\mathcal H$ 
is defined by models of Theorem \ref{THM:ReductionRiccati}, for convenient analytic coordinates $x_i,y_i$;
the restriction $|D|\cap V_i$ is defined by $\{ f_i=0\}$ with $f_i=x_i$ or $f_i=x_i y_i$. 
We can assume that intersections $V_i\cap V_j$ are simply connected and do not contain singular points of the polar divisor
$(\mathcal H)_\infty$. We want to prove that, on $V_i\cap V_j$, we have
$$f_i=a_{ij}f_j+b_{ij}f_j^{k+1}$$
with constant $a_{ij}^{2k}=1$ and $b_{ij}$ holomorphic; moreover, in the split case, we want to show that we can choose 
$f_i$'s such that $a_{ij}^{k}=1$. This will prove the Proposition.

We are going to work at the formal completion of $V$ along $D$. Over each intersection $V_i\cap V_j$, 
we can choose formal trivializing coordinates for the bundle $z_i,z_j:P\vert_{V_i\cap V_j}\to\mathbb P^1$ such that the foliation 
$\mathcal H$ is defined by the closed formal  $1$-forms
\[
\Omega_i = \frac{dz_i}{z_i} + \left( \lambda\frac{df_i}{f_i} + \frac{df_i}{f_i^{k+1}} \right)\ \ \ \text{and}\ \ \ 
\Omega_j = \frac{dz_j}{z_j} + \left( \lambda\frac{df_j}{f_j} + \frac{df_j}{f_j^{k+1}} \right).
\]
Indeed, if $\mathcal H$ has model (\ref{ModelIrreg}) over $V_i$, then this follows from 
\S \ref{SecReview1D} (see (\ref{eq:FormalNormFormUnramified})); on the other hand, for the model (\ref{ModelIrreg0}), with $f_i=x_i$ or $x_iy_i$, then we have 
\[
\frac{dz}{z}+\frac{df_i}{f_i^{k+1}}+\lambda_x\frac{dx_i}{x_i}+\lambda_y\frac{dy_i}{y_i}=
\left(\frac{dz}{z}+\tilde\lambda_y\frac{dy_i}{y_i}\right)+\frac{df_i}{f_i^{k+1}}+\lambda_x\frac{df_i}{f_i}
\]
and we get the normal form $\Omega_i$ by setting $z_i=z\cdot\exp(\tilde\lambda_y\log(y_i))$ 
for a convenient determination of the logarithm (note that $y_i\not=0$ on $V_i\cap V_j$). 
We do exactly the same over $V_j$.

We claim that we can choose $\Omega_i=\pm\Omega_j$ over $V_i \cap V_j$. Indeed,
since $\Omega_i$ and $\Omega_j$ define the same foliation, they must be proportional: we have
$\Omega_i=g_{ij}\Omega_j$ for some (formal) function $g_{ij}$. 
From closedness condition for these $1$-forms, 
we deduce the $g_{ij}$ must be a first integral for $\mathcal H$;
therefore, $g_{ij}$ is constant. But since residues of $\Omega_i$ and $\Omega_j$ 
on non vertical poles $z_i=0,\infty$ are equal to $\pm1$, we thus get $g_{ij}=\pm1$.

We note that the non vertical poles of $\Omega_i$ define two formal sections over $V_i\cap V_j$ 
that coincide over $\vert D\vert\cap V_i\cap V_j$ with the singular locus of $\mathcal H$.
In the split case, the non vertical poles therefore globally define two sections so that we can
globally fix the residues of all $\Omega_i$'s on these. Consequently, in the split case,
we can moreover assume $\Omega_i=\Omega_j$ over $V_i \cap V_j$.

If $\lambda\neq0$, then $\Omega_i=\Omega_j$ because of the sign of the residue over $D$.
Therefore, in the unsplit case, we necessarily have $\lambda=0$.

Let us first deal with the split case, therefore assuming $\Omega_i=\Omega_j$.
We must have  $f_i=f_{ij}(x_j,y_j)f_j$ and $z_i=g_{ij}(x_j,y_j) z_j$ for some invertible functions
$f_{ij}$ and $g_{ij}$ on $V_i\cap V_j$.  Then we get
\[
 {0=\Omega_i -  \Omega_j  =} \frac{dg_{ij}}{g_{ij}} + \lambda \frac{d f_{ij}}{f_{ij}} + \frac{1}{f_j^k} \frac{df_{ij}}{f_{ij}^{k+1}} + \left( \frac{1}{f_{ij}^k} - 1 \right) \frac{df_j}{f_j^{k+1}}
\]
We deduce that the $1$-form 
\[
\frac{1}{f_j^k} \frac{df_{ij}}{f_{ij}^{k+1}} + \left( \frac{1}{f_{ij}^k} - 1 \right) \frac{df_j}{f_j^{k+1}}
= \frac{1}{f_{ij}^{k+1}} \underbrace{\left( \frac{df_{ij}}{f_j^k} + \left( f_{ij} - {f_{ij}^{k+1}}  \right) \frac{df_j}{f_j^{k+1}} \right)}_{\Theta} \, .
\]
cannot have poles.

Notice that the residue of $f_j^k \Theta$ along $\{f_j=0\}$ is nothing but
$
f_{ij} - f_{ij}^{k+1} \mod f_j .
$
Therefore $f_{ij}^k = 1 \mod f_j$. 
Now, let us write $f_{ij}=a_{ij}(1+b_{ij}f_j^n)$ for some $n>0$ and $b_{ij}$ holomorphic (and $a_{ij}^k=1$).
Then we get
$$\Theta=\frac{(n-k)a_{ij}b_{ij}f_j^ndf_j+o(f_j^n)}{f_j^{k+1}}.$$
If $n<k$, then $b_{ij}$ must vanish along $D$ and we can set $f_{ij}=a_{ij}(1+\tilde b_{ij}f_j^{n+1})$.
By induction, we arrive at $n=k$. This  establishes the proposition in the split case.

Finally, in the unsplit case, whenever we have $\Omega_i=-\Omega_j$, we can replace 
$\Omega_i$ by $\tilde\Omega_i=-\Omega_i$, and 
$(f_i,z_i)$ by $(\tilde f_i,\tilde z_i)=(af_i,1/z_i)$
with $a^{k}=-1$, so that we are back to the previous discussion: we deduce that 
$$f_i=\frac{\tilde a_{ij}f_j+\tilde b_{ij}f_j^{k+1}}{a}=a_{ij}f_j+b_{ij}f_j^{k+1}$$
with $a_{ij}^{2k}=\left(\frac{a_{ij}}{a}\right)^{2k}=1$, ending the proof.
\end{proof}

\begin{remark}\label{rmk:FlatCoord}
In the system of ``transverse coordinates'' $f_i:V_i\to \C$ constructed in the course of the proof,
the Riccati foliation $\mathcal H\vert_{V_i}$ is, except at points of type (\ref{ModelIrreg0}), 
locally defined by
\[
\Omega_i = \frac{dz_i}{z_i} + \left( \alpha_i(f_i) z_i +  \lambda\frac{1}{f_i} + \frac{1}{f_i^{k+1}}  + \beta_i(f_i)+\frac{\gamma_i(f_i)}{z_i} \right)df_i   \,
\]
for holomorphic bundle coordinate $z_i$, where $\alpha_i,\beta_i,\gamma_i$ are holomorphic functions of $f_i$.
We can assume the functions $\alpha_i,\beta_i,\gamma_i$ vanishing at arbitrary high order along  $\vert D\vert$.
Recall from Section \ref{S:stokes} that the obstruction to make them just zero is given by the non triviality of Stokes matrices.
\end{remark}

\begin{remark}\label{rmk:representationIrregular}
On $V_i\cap V_j$ we get
$$
\left\{\begin{matrix}f_i=a_{ij}(f_j+b_{ij} f_j^{k+1})\\
z_i=c_{ij} z_j^{\pm1}\hfill\end{matrix}\right.
$$
with $a_{ij}^k=\pm1$ (same sign) and $b_{ij},c_{ij}\in\mathcal O_{V_i\cap V_j}$
satisfying $c_{ij}\exp(b_{ij}a_{ij}^k)\vert_{\vert D\vert}\equiv 1$.
Considering transition multiplicators $a_{ij}$, we have an induced  representation $\psi: \pi_1(|D|) \to \C^*$ taking values in the group of (respective) roots of the unity, which describe how the differentials $df_i$ change when we follow closed paths  along $D$.
\end{remark}

\subsection{Smooth fibration}
We have just proved that a connected component $D$ of the irregular divisor $I$
has local equations $f_i^{2k}:V_i\to \C$ that patch together up to order $k$, where $k$ is the irregularity index.
In the setting of real $C^\infty$ functions, we can modify  these local equations so that they can patch together to define a fibration
on $V$ having $D$ as a singular fiber and which is smooth elsewhere.

\begin{lemma}\label{lem:CinftyFibration}
Let $(P,\mathcal H)$ be a reduced Riccati foliation on a projective surface $S$.
Let $D$ be a connected component of the irregular divisor $I$ like in Proposition \ref{P:flat}
and $r$ be the torsion order of $D_{red}$.
Then, there exists a $C^\infty$ map $f:V\to\mathbb D$ defined on an analytic neighborhood $V$
of $\vert D\vert$ such that:
\begin{itemize}
\item $f$ induces a $C^\infty$ locally trivial fibration over $\mathbb D^*=\mathbb D\setminus\{0\}$,
\item $f$ coincides with local equations $f_i^{r}:V_i\to\C$ of Remark \ref{rmk:FlatCoord} up to order $k$.
\end{itemize}
\end{lemma}
\begin{proof}
The transition functions $\{f_{ij} \in \mathcal O_V^* (V_i\cap V_j)\}$ of the proof of Proposition \ref{P:flat}
define an element in $H^1(V,\mathcal O_V^*)$ which corresponds to  {$\mathcal O_V(D)$}.
Let  $\mathcal A_V$ and $\mathcal A_V^*$ denote, respectively,
the sheaves of $C^{\infty}$ functions and of  $C^{\infty}$ invertible functions. Notice that  $H^1(V,\mathcal A_V^*)$ is isomorphic to $H^2(V,\mathbb Z)$,
as $\mathcal A_V$ has no cohomology in positive degree ($\mathcal A_V$ is a fine sheaf). Notice also that  the restriction morphism from $H^2(V,\mathbb Z)$
 to $H^2(|D|,\mathbb Z)$ is an isomorphism since the inclusion of $|D|$ in $V$ is a homotopy equivalence.

Since ${f_{ij}^{r}}\vert_{|D| \cap V_i}=1$, the element $\{ f_{ij}^{r} \} \in H^1(V,\mathcal O_V^*)$  maps to the trivial element of $H^1(V,\mathcal A_V^*)  \simeq H^2(V,\mathbb Z) \simeq H^2(|D|,\mathbb Z)$.
Thus  we can find non-vanishing $C^{\infty}$-functions $\{ g_i \in \mathcal A_V^*(V_i)\}$ such that
$
f_{ij}^{r} = \frac{g_i}{g_j} \, .
$
Moreover, we can choose the functions $g_i$ satisfying ${g_i}\vert_{|D| \cap V_i} = 1$.
Our assumptions imply that we can further assume that the function $f_{ij}^{r}$ are constant equal to one when restricted  to $D$.

Therefore we can define a  $C^{\infty}$ function $f: V \to \C$ by the formulas $f\vert_{V_i} = (f_i)^{r}/g_i$.
The function $f$ clearly satisfies $f^{-1}(0)=|D|$ set-theoretically. We claim that, after perhaps shrinking  $V$,
the  critical set of $f$  is contained in $|D|$. At a neighborhood
of a smooth point of $|D|$ the function $f$ is a power of a submersion. At a neighborhood of a singular point of $|D|$, the function $f$ is of the form
$ h(x,\overline x, y, \overline y) x y$ and therefore
$
df =  xy d h + h( y dx +  x dy ) \, .
$
Since this expression  has an isolated singularity at zero (the singular point of $|D|$) it follows  that the critical  set of $f$ is indeed contained in $D$.
Replacing $V$ by $f^{-1}(\mathbb D_{\varepsilon})$ for a sufficiently small $\varepsilon$ we have just proved  the existence of a $C^{\infty}$ proper map $f: V \to \mathbb D_{\varepsilon}$
from $V$ to the  disk of radius $\varepsilon$ which maps $|D|$ to the origin in $\mathbb D_{\varepsilon}$ and, when
restricted to $V-|D|$, becomes a locally trivial fibration over $\mathbb D_{\varepsilon}^*$. 
\end{proof}

\subsection{Closed $1$-form and Riccati foliations}\label{sec:ClosedFormStokesBis}
We present now a semi-local version of Lemma \ref{lem:ClosedFormStokes}

\begin{prop}\label{prop:ClosedFormStokes}
Let $(P,\mathcal H)$ be a reduced Riccati foliation on a projective surface $S$.
Let $D$ be a connected component of the irregular divisor of $(P,\mathcal H)$
and $V$ be a sufficiently small neighborhood of $\vert D\vert$.
Assume we are in the split case. Then the following assertions are equivalent:
\begin{itemize}
\item $\mathcal H\vert_{V}$ is defined by a meromorphic closed $1$-form on the restriction $P\vert_{V}$,
\item there are two analytic sections $V\to P$ of the bundle that are $\mathcal H$-invariant,
\item at some point $p\in\vert D\vert$ (in fact any), the local model for $\mathcal H$ is of type {\bf Irreg${}^{0}$},
or of type {\bf Irreg} with trivial Stokes (see Theorem \ref{THM:ReductionRiccati}).
\end{itemize}
\end{prop}

\begin{proof}
If the local model for $\mathcal H$ has trivial Stokes at $p$ (which is automatic in the first item,  see Lemma \ref{lem:ClosedFormStokes}),
then it has two invariant
analytic sections. By using local trivializations of $\mathcal H$ given in Section \ref{sec:locfac} (see also Remark \ref{rmk:FlatCoord}), we see that this property
propagates all along $\vert D\vert$. By the way, we get a two-fold section of the bundle; since they
are locally attached to the two irreducible components of $Z$, we get actually two distinct global $\mathcal H$-invariant sections.
It follows from Lemma \ref{lem:MultisectionH} that $\mathcal H$ is defined by a meromorphic closed $1$-form.
Finally, it follows from Lemma \ref{lem:ClosedFormStokes} that, if defined by a closed $1$-form, then Stokes are
trivial at any point $p\in\vert D\vert$.
\end{proof}

\begin{prop}\label{prop:ClosedFormIrregRiccatiFiniteCover}
Let $(P,\mathcal H)$ be an irregular Riccati foliation on a connected complex surface $S$.
If there exists a generically finite morphism $f:\tilde S\to S$ such that $f^*(P, \mathcal H)$
is defined by a closed rational $1$-form, then there exists another one $f':\tilde S'\to S$ of degree at most two
with the same property.
\end{prop}
\begin{proof}The pull-back $(\tilde P,\tilde{\mathcal H}):=f^*(P, \mathcal H)$ must be also irregular; 
in particular, $\tilde{\mathcal H}$ does not admit non constant rational first integral.
By Proposition \ref{prop:RiccClosedForm}, we deduce that, up to a birational bundle transformation,
$\tilde{\mathcal H}$ is defined by a Riccati of the form $\frac{dz}{z}+\omega=0$ for a closed
rational $1$-form $\omega$ on $S$ (note that $dz+\omega=0$ is regular). In particular,
the two sections $z=0,\infty$ are $\tilde{\mathcal H}$-invariant. Pushing them down via $f$,
we get a $n$-section of $P$ which is $\mathcal H$-invariant, $n\ge2$. 
We can apply Lemma \ref{lem:MultisectionH}: by irregularity, $\mathcal H$ cannot have
non constant meromorphic first integral and is therefore defined by a closed rational $1$-form 
maybe passing to a two-fold cover $X'\to X$.
\end{proof}

\subsection{Monodromy around the irregular divisor}

\begin{prop}\label{P:monodromia irregular}
Let $(P,\mathcal H)$ be a reduced Riccati foliation on a complex surface $S$.
Let $D$ be a connected component of the irregular divisor $I$ of the form $D=k D_{red}$.
Then, there is a neighborhood $V$ of $|D|$ in $S$ in restriction to which
the monodromy of $(P,\mathcal H)$ is virtually abelian. More precisely,
at least one of the following assertions holds true.
\begin{enumerate}
\item\label{ItemClosedCase} Maybe after passing to an \'{e}tale covering $\tilde V\to V$ of degree two,
the (pulled-back) Riccati foliation $\mathcal H$ is defined by a closed meromorphic $1$-form 
over $\tilde V$ (and the monodromy is abelian).
\item\label{ItemFactorCase} The monodromy factors through a $C^{\infty}$-fibration $f: V - D \to \mathbb D^*$ 
and is therefore cyclic.
\end{enumerate}
If the supports of $(\mathcal H)_{\infty}$ and  $D$ do not coincide on $V$, then we are in case {\rm (1)}.
\end{prop}
\begin{proof}
Let $V$ be a small tubular neighborhood of $D$, and let  $R: V \to |D|$
be a deformation retract. Since the statement is local, from now on $\mathcal H$ will be seen as a Riccati foliation defined
over $V$. 
Maybe passing to an \'{e}tale covering $\tilde V\to V$ of degree two, we may assume 
we are in the split case:
 the singular locus $\sing(\mathcal H)$ consists in two disjoint sections $D\to P\vert_D$,
and fibers over singular points of $D$.

If all the  Stokes matrices along $D$ are trivial, then Proposition \ref{prop:ClosedFormStokes}
implies that we are in case (1) of the statement.
If $(\mathcal H)_{\infty}$ and $D$ have distinct supports in $V$,
this means that $(\mathcal H)_{\infty}$ contains at least one logarithmic component,
intersecting $|D|$ at some point $p$;
the model at $p$ is (\ref{ModelIrreg0}) in Theorem \ref{THM:ReductionRiccati}, 
and all Stokes are therefore trivial along $D$.

Let us now assume that Stokes matrices are non trivial at a smooth point $x_0 \in D$.
Let $\Sigma$ be a germ of curve transverse  to $D$ at $x_0$.
Let $\pi:P\to V$ be the natural projection. The restriction of
$\mathcal H$ to $\pi^{-1}(\Sigma)$,  is a Riccati foliation over $\Sigma$ with an invariant fiber $\{ x_0\} \times \P^1$ having two saddle-nodes over it. As explained in Section \ref{S:stokes}, if $k$ is the order of $D$ at $x_0$ then there are $2k$ closed sectors on $\Sigma$,
such that over the interior of each of them, the Riccati foliation is analytically conjugated to $\frac{dz}{z} + \frac{dx}{x^{k+1}}+\lambda\frac{dx}{x}$  and the
conjugation extends continuously to the boundary. Over each of these sectors there are exactly two leaves with distinguished topological behavior:
the closure of each of these distinguished leaves
intersect the central fiber $\{ x_0\} \times \P^1$  at a unique point. 

\begin{figure}
\begin{center}
\begin{tikzpicture}
    \begin{scope}
      \draw (150:3) node {$Re(x^3)=0$}    ;
     \draw  (40:1.9) node {$\mathscr S_{\Sigma}$}    ;
      \draw (0,-2.7) node {The  sector $\mathscr S_{\Sigma}$.}    ;
    \foreach \angle in {30,90,150,210,270,330}
      \draw[dashed] (0,0) -- (\angle :2.5);

    \foreach \angle in { 150 }
     \path [draw ,fill= lightgray,semitransparent] (0,0) -- (\angle-40:1.5) arc (\angle-40:\angle + 40 :1.5) -- (0,0);

    \end{scope}

    \begin{scope}[shift={(6,0)}]

      \draw (0,-2.7) node {The intersection of   $\mathscr S$ and $\Sigma$.}    ;
    \foreach \angle in {30,90,150,210,270,330}
      \draw[dashed] (0,0) -- (\angle :2.5);

    \foreach \angle in {30,150,270 }
     \path [draw ,fill= lightgray,semitransparent] (0,0) -- (\angle-40:1.5) arc (\angle-40:\angle + 40 :1.5) -- (0,0);

    \end{scope}
\end{tikzpicture}
\caption{}
\end{center}
\end{figure}

Let $\mathscr S_{\Sigma}$ be the interior
of one of these   sectors. The local triviality of $\mathcal H$ along the smooth part of $D$, and the local normal form of $\mathcal H$ at the singularities of $D$, allow us to construct an open set $\mathscr S \subset V - |D|$ which  extends $\mathscr S_{\Sigma}$
and over which we never lose sight of the two distinguished leaves of the restriction of $\mathcal H$ to $\pi^{-1}(\mathscr{S}_{\Sigma})$. This open subset can be indeed realized as the saturation of $\mathscr S_{\Sigma}$ by the fibers 
of the fibration $f\vert_{V - |D|} : V-|D| \to \mathbb D^*$  defined by Lemma \ref{lem:CinftyFibration}. The  intersection of the resulting open set $\mathscr S$ and the initial transversal $\Sigma$
has $r = \# \psi(\pi_1(|D|))$ distinct connected components,  where $\psi:\pi_1(|D|) \to \C^*$ is the unitary representation
defined in Remark \ref{rmk:representationIrregular}.

 The restriction of the monodromy representation of $\mathcal H$ to $\mathscr S$ has its image (up to conjugation)
contained in  $\C^*$ as $\mathcal H$ over $\mathscr S$   has two distinguished leaves which are not permuted
because of the assumption on $\mathrm{sing}(\mathcal H)$. By construction, the set $\mathscr S$ has the form $f^{-1} (S)$ with $S$ a contractible open set of ${\mathbb D}^*$. Therefore, it has the same homotopy type as the general fiber. 
In particular,
$\pi_1(\mathscr S)$ is a normal subgroup of $\pi_1(V-|D|)$ thanks to the vanishing of $\pi_2({\mathbb D}^*)$. If we do this construction choosing base points at two adjacent sectors with
non trivial Stokes transition matrix, we conclude that indeed the monodromy of $\mathcal H$ on $\mathscr S$ is trivial since a
nontrivial Stokes matrix at one hand conjugates the corresponding monodromy representations, and at the other hand it does not
respect the fixed points of the two monodromy representations. We conclude that in the presence of a nontrivial Stokes matrix
the monodromy factors through $f_*: \pi_1(V-|D|) \to \pi_1(\mathbb D^*)$ as wanted.
\end{proof}

\begin{lemma}\label{L:finalpiece}
Let $(P,\mathcal H)$, $D$, and $k$ be as in the statement of Proposition \ref{P:monodromia irregular}
and assume that we are not in case (\ref{ItemClosedCase}).
If the normal bundle of $D_{red}$ has torsion order $\ge k$,
then the monodromy of $(P,\mathcal H)$ is non-trivial.
\end{lemma}
\begin{proof}
Notation as in the proof of Proposition \ref{P:monodromia irregular}. 
Consider the unitary representation $\psi:\pi_1(|D|) \to \C^*$
defined in Remark \ref{rmk:representationIrregular};  
the torsion order $r$ of the normal bundle of $D$ satisfies $r = \# \psi(\pi_1(|D|))$.
By Proposition \ref{P:flat}, we know that $r$ divides $k$ (resp. $2k$) in the split (resp. unsplit) case. 
Since $r\ge k$, we have $r=k$ or $2k$.

Consider first the split case: this obviously implies $r=k$. 
Let $\gamma \in |D|$ a loop such that $\psi(\gamma)= \exp(2\pi i /k)$
We can lift $\gamma$ to a path $\tilde{\gamma}$ in a region $\mathscr S$ such that the initial and final points lie in two consecutive
small sectors. Moreover, since $\mathcal H$ is not given by a closed meromorphic $1$-form we can assume that at least one of the two distinguished leaves over the initial
sector is not a distinguished leaf for the final sector. Closing the path $\tilde{\gamma}$ using an arc on the transversal $\Sigma$,
we obtain a path in $V-|D|$ with non-trivial
monodromy.

Consider now the unsplit case.
Going back to the proof of Proposition \ref{P:flat}, we see that there must be some 
loop $\gamma \in |D|$ such that $\psi(\gamma)^k= -1$. This shows that $r=k$ is impossible,
thus $r=2k$ in this case. After a two-fold \'etale covering $\tilde V\to V$, we are in the previous split case
with normal bundle having torsion $k$ coinciding with irregularity index. We thus obtain a loop
$\tilde\gamma$ with non trivial monodromy on $\tilde V$; its projection on $V$ has the same property.
\end{proof}

\section{Structure}\label{S:structure}

\subsection{Proof of Theorem \ref{THM:C}}
Let $(P,\mathcal H)$ be a Riccati foliation over the projective manifold $X$ with irregular singular points.
Assuming that it is not defined by a closed $1$-form after a two-fold covering, our aim is to prove that it factors through a curve. 

By Proposition \ref{prop:ClosedFormIrregRiccatiFiniteCover},
we may assume that $f^*(P,\mathcal H)$ is not defined by a closed rational $1$-form 
for any generically finite dominant rational map $f:X'\dashrightarrow X$.
According to Proposition \ref{P:reduction2}, it suffices to prove the factorization of $(P,\mathcal H)$
in restriction to a surface $S\subset X$ obtained as an intersection of general hyperplane sections. According to Proposition \ref{P:reduction2closed} (and an obvious induction on the dimension),
we may assume that the restriction is not defined by a closed rational $1$-form after any 
generically finite dominant rational map $f:S'\dashrightarrow S$.
Furthermore, Proposition \ref{P:pushRiccati}  allows us  to assume that the singularities of $(P,\mathcal H)$ are as in the conclusion of Theorem \ref{THM:ReductionRiccati}. 

Let $D$ be a connected component of the irregular divisor $I$: we have $D=kD_{red}$.
It follows from Proposition \ref{P:flat} that $D$ has torsion normal bundle, and from Proposition \ref{P:monodromia irregular}
that the monodromy of $(P,\mathcal H)$ is virtually abelian at the neighborhood $V$ of $D$.

{\it First case: the global monodromy of $(P,\mathcal H)$ on $S$ is not virtually abelian.}
In particular, it is strictly larger than the local monodromy around $D$: 
$$\rho(\pi_1(S-|(\mathcal H)_\infty|))\ \not=\ \rho(\pi_1(V-|(\mathcal H)_\infty|)).$$
Then arguments used in the proof of Theorem \ref{T:criteriofactor} show that $D$ is the fiber of a fibration
$f:S\to C$ (we use extra topology in $S-V$ to construct a ramified cover with several disjoint copies of $D$
and then apply Theorem \ref{T:Totarobis}). Let $U$ a dense open Zariski subset of $S-|(\mathcal H)_\infty|$ such that $f_{|U}$ is a topologically a locally trivial fibration on its image\footnote{Beware that the general fiber of $f_{|U}$ is not necessarily compact, especially when $D$ intersect a logaritmic component of the polar locus $(\mathcal H)_\infty$.}. As in the proof of Theorem \ref{T:criteriofactor}, the monodromy of a general fiber of $f_{|U}$ is a normal subgroup
of the global monodromy group and is therefore trivial. Thus, the monodromy representation factors through $f$.
But the general fiber of $f$ does not intersect
the irregular divisor and we can apply Proposition \ref{P:pullRiccati} to conclude that the Riccati foliation
factors as well.

We can now assume that the global monodromy is virtually abelian, and after passing to a finite covering,
that it is  torsion free. Thus the global monodromy either is abelian and infinite, or trivial.

{\it Second case: the global monodromy of $(P,\mathcal H)$ on $S$ is abelian and infinite.}

We first aim to prove that $D$ is a fiber of a fibration $S\to C$ (up to finite cover).
If the global monodromy is strictly larger than the local monodromy around $D$,
we can argue as in the first case: we can use the extra monodromy to produce a finite cover
on which we have several copies of $D$ providing a fibration. Assume now
$$\rho(\pi_1(S-|(\mathcal H)_\infty|))\ =\ \rho(\pi_1(V-|(\mathcal H)_\infty|)).$$
If the Riccati foliation $\mathcal H$ nearby $D$ satisfies conclusion (1) of Proposition \ref{P:monodromia irregular}, then there is a $2$-section $V\dashrightarrow P\vert_V$ which is invariant by $\mathcal H$,
and in particular by the monodromy. This $2$-section therefore extends outside the extra poles 
of $\mathcal H$. If the Riccati foliation has only logarithmic poles outside of $D$, then the $2$-section
extends on the whole of $S$, proving that $\mathcal H$ is given by a closed $1$-form on a two-fold cover
(see  Lemma \ref{lem:MultisectionH}), contradiction. If, on the contrary, $\mathcal H$ has 
an extra irregular pole, say $D'$, disjoint from $D$, both of them having torsion normal bundle by 
Proposition \ref{P:flat}; then, we get a fibration by Theorem \ref{T:Totarobis}.

In conclusion (2) of Proposition \ref{P:monodromia irregular}, the monodromy factors, around $D$,  
through the $C^{\infty}$-fibration $f:V-D\to\mathbb D^*$: the monodromy is infinite cyclic. If the representation takes values in $\mathbb G_a$  we  compose a suitable multiple of it with an exponential so that we can assume that it takes  values in $\mathbb G_m$ and is still infinite.   
Consider the Deligne logarithmic flat connection $(L,\nabla)$
realizing this representation. By construction, the residue of $\nabla$ along an irreducible component $E_i$ 
of the polar divisor takes the form $k_i\lambda+n_i$, where $\lambda$ is an irrational number (the monodromy is infinite) and $k_i,n_i$ are integers. 
The residue formula gives 
$$0=\sum_i\mathrm{Res}_{E_i}c_1(E_i)\ +\ c_1(L)=\left(\sum_ik_ic_1(E_i)\right)\lambda\ +\ \left(\sum_in_ic_1(E_i)\ +\ c_1(L)\right).$$
Since $\lambda$ is irrational, each term in parenthesis is zero and we get in particular
$$0=\sum_ik_ic_1(E_i)=c_1(D)-c_1(E)$$
where $D$ is our connected component of the irregular divisor, and $E$ is a divisor  disjoint
from $D$ and contained in  $\vert(\mathcal H)_\infty\vert$. After splitting $E=E_+-E_-$
with $E_+$ and $E_-$ effective, we know that $E_+$ and $E_-$ have non positive
self-intersection by Hodge index Theorem. On the other hand, the equality
$$0=D\cdot D=E\cdot E=E_+\cdot E_+ - 2 E_+\cdot E_- + E_-\cdot E_-$$
forces each term, {\it a priori} $\le0$, to be zero. Finally, again by Hodge 
index Theorem, $E_+$ and $E_-$ must be disjoint and have Chern classes  proportional to
the Chern class of $D$. We can apply Theorem \ref{T:Totarobis} to produce a fibration.

We have just proved, in the case monodromy is   abelian and infinite, that $D$ is fiber of a holomorphic fibration $f:S\to C$ (we keep the same notation as for the $C^{\infty}$-fibration).
It remains to show that the Riccati foliation $\mathcal H$ factors. When the monodromy
is trivial along a generic fiber of $f$, this clearly follows from Proposition \ref{P:pullRiccati}.
If not, we have infinite monodromy along fibers and, by looking at a fiber close to $D$, we must be in case (1)
of Proposition \ref{P:monodromia irregular}. In particular, there is a $2$-section of the $\P^1$-bundle $P$ that are invariant by the Riccati foliation $\mathcal H$. In fact, since the monodromy is torsion-free,
we are in the split case: the $2$-section splits into two disjoint sections.
These two sections are tangent to the Riccati foliation $\mathcal H$ and also to the pull-back to $P$ of the foliation defined by the fibration. We obtain two
curves in the Hilbert scheme of $P$. Since tangency to $\mathcal H$ imposes a closed condition on the Hilbert scheme of $P$,
these two curves have Zariski closure of dimension one (recall that over each fiber of $f$ we have exactly two sections of $P$). Thus they spread two surfaces, sections of $P$, which are invariant by $\mathcal H$.
By Lemma \ref{lem:MultisectionH}, $\mathcal H$ is defined by a closed $1$-form, contradiction.

{\it Third case: the monodromy of the Riccati foliation $\mathcal H$ is trivial.}
Like in the previous case, if $\mathcal H$ is defined by a closed $1$-form at the neighborhood of $D$,
then it extends as a global $1$-form, except if there is another irregular polar component, in which case
we get two disjoint divisors with torsion normal bundle (see Proposition \ref{P:flat})
and therefore a fibration (see Theorem \ref{T:Totarobis}) through which $\mathcal H$ factors (see Proposition \ref{P:pullRiccati}).
On the other hand, if we are in case (2) of Proposition \ref{P:monodromia irregular},
then Lemma \ref{L:finalpiece} implies that $r < k$, i.e. the order of the normal bundle of $D_{red}$ is strictly smaller
than the multiplicity of the irregular divisor. The existence of a fibration with a fiber supported on $|D|$ follows
from Theorem \ref{T:Neeman}, and the Riccati foliation factors according to Proposition \ref{P:pullRiccati}.
\qed

\subsection{Proof of Theorem \ref{THM:E}}
Let $(E,\nabla)$ be a flat meromorphic $\mathfrak{sl}_2$-connection on $X$.
In the case $(E,\nabla)$ is regular, then the conclusion of Theorem  \ref{THM:E}
directly follows from Corollary \ref{COR:B}. Indeed, the Riemann-Hilbert correspondance
establishes a one-to-one correspondance between representations up to conjugacy 
and regular connections up to birational bundle tranformations (see \cite{Deligne}). For instance, 
if the monodromy is virtually abelian, i.e. abelian after a finite cover, then it is 
either diagonal, or unipotent, and can be realized by one of the two models 
of case (1) in Theorem  \ref{THM:E}; the Riemann-Hilbert correspondance
provides the birational equivalence. Similarly, cases (1) and (2) of Corollary \ref{COR:B}
for the monodromy respectively imply cases (2) and (3) of Theorem  \ref{THM:E}
for the regular connection. Let us now assume $(E,\nabla)$ irregular.

Let us consider $P:=\P (E)$ the $\P^1$-bundle associated to $E$; horizontal section of $\nabla$
induce a Riccati foliation $\mathcal H$ on $\pi:P\to X$ which is irregular by assumption.
We can apply Theorem \ref{THM:C} to the projective connection/Riccati foliation $(P,\mathcal H)$
and discuss the  two possible conclusions.

Assume first that $\mathcal H$ is defined by a closed $1$-form (maybe passing to a finite covering of $X$).
After birational bundle transformation, we can assume $P_0=X\times\P^1$ and $\mathcal H_0$
defined by
$$\Omega_0= \frac{dz}{z}+2\omega\ \ \ \text{or}\ \ \ \ \Omega_0 = dz+\omega$$
with $\omega$ a closed $1$-form on $X$ (see Proposition \ref{prop:RiccClosedForm}).
These Riccati foliations are   induced by those explicit connections of Corollary \ref{THM:E} (1):
$$\nabla_0=d+\begin{pmatrix}\omega&0\\ 0&-\omega\end{pmatrix}\ \ \ \text{or}\ \ \ \nabla_0= d+\begin{pmatrix}0&\omega\\ 0&0\end{pmatrix}$$
on the trivial bundle $E_0:=\mathcal O_X\oplus\mathcal O_X$.
There is a birational bundle trivialization $E\dashrightarrow E_0$ making commutative the diagram
$$\xymatrix{
    E \ar@{.>}[r]^\phi \ar[d] & E_0 \ar[d] \\
    P \ar@{.>}[r]  & P_0
  }
  $$
Obviously, $\phi_*\nabla$ is projectively equivalent to one of the models $\nabla_0$:
$\phi_*\nabla=\nabla_0\otimes\zeta$
with $(\mathcal O_X,\zeta)$ a flat rank one connection over $X$, birationally equivalent to the trivial connection by construction.
This means that one can write $\zeta=d+\frac{df}{f}$ and, maybe tensoring by the logarithmic connection $(\mathcal O_X,d+\frac{1}{2}\frac{df}{f})$
(whose square has trivial monodromy), we get equality $\phi_*\nabla=\nabla_0$. Note that, passing to the $2$-fold covering defined
by $z^2=f$, the connection $\nabla$ is birationally gauge equivalent to $\nabla_0$ (without tensoring).

Assume now that $(P,\mathcal H)$ is birationally gauge equivalent to the pull-back $f^*(P_0,\mathcal H_0)$
of a Riccati foliation over a curve, $f:X\dashrightarrow C$ with $P_0=C\times\P^1$. Denote by $\nabla_0$
the unique $\mathfrak{sl}_2$-connection on the trivial bundle $E_0$ over $C$ inducing the projective connection $(P_0,\mathcal H_0)$.
Then $(E,\nabla)$ is birationally equivalent to $f^* (E_0,\nabla_0) \otimes ( \mathcal O_X ,\zeta)$ with $\zeta$ logarithmic rank one connection  having
  monodromy in the center of $\SL(\C)$.
\qed

\section{Transversely projective foliations}\label{SecTransvProj}

\subsection{Basic Lemma}\label{SecBasicLemma}
Let $\mathcal F$ be a transversely projective foliation on a projective manifold $X$
defined by a triple $(\omega_0,\omega_1,\omega_2)$.
The Riccati foliation $\mathcal H$ defined on $X\times\P^1$ by
$$dz+\omega_0+z\omega_1+z^2\omega_2=0$$
is integrable (\ref{IntCondOmega})
$$\left\{\begin{matrix}
d\omega_0=\hfill\omega_0\wedge\omega_1\\
d\omega_1=2\omega_0\wedge\omega_2\\
d\omega_2=\hfill\omega_1\wedge\omega_2
\end{matrix}\right.$$
and the foliation $\mathcal F$ is defined by restricting $\mathcal H$ to the section $\sigma:X\to P$ given by $z=0$.
Another projective triple $(\omega_0',\omega_1',\omega_2')$ defines the same
transversely projective foliation, i.e. the same
foliation $\mathcal F$ with the same collection of local first integrals at a general point of $X$ if,
and only if, there are rational functions $a,b$ on $X$ such that $a\not\equiv0$ and
\begin{equation}\label{ChangeTriple}
\left\{\begin{matrix}
\omega_0'=a\omega_0\hfill\\
\omega_1'=\omega_1-\frac{da}{a}+2b\omega_0\hfill\\
\omega_2'=\frac{1}{a}\left(\omega_2+b\omega_1+b^2\omega_0-db\right)
\end{matrix}\right.
\end{equation}
(see \cite{Scardua}).
This exactly means that the Riccati foliation $\mathcal H'$ defined by $\Omega'=d\tilde z+\omega_0'+\tilde z\omega_1'+\tilde z^2\omega_2'$
is derived by gauge transformation $\frac{1}{z}=a\frac{1}{\tilde z}+b$. 
The following is well known (cf \cite[Chap. II, Prop. 2.1, p.193]{Scardua}
and \cite[Lemma 2.20] {Croco2}).

\begin{lemma}\label{LemFClosedHClosed}
If $\mathcal F$ is defined by a closed $1$-form $\omega$, $d\omega=0$, then up to gauge transformation (\ref{ChangeTriple}), 
the Riccati foliation $\mathcal H$ is defined by 
$$\Omega=dz+\omega(1+\phi z^2)\ \ \ \text{where}\ \phi\ \text{is a first integral for}\ \mathcal F.$$
\end{lemma}
\begin{proof}One can first write $\omega=a\omega_0$ for some rational function $a$ and use 
gauge transformation (\ref{ChangeTriple}) to set $(\omega_0,\omega_1,\omega_2)=(\omega,0,\omega_2')$.
Indeed, once $\omega_0=\omega$ is closed, we deduce from (\ref{IntCondOmega}) that $\omega_0\wedge\omega_1=0$,
and therefore $\omega_1=b\omega_0$
for some rational function $b$; consequently, $\omega_1$ can be killed by gauge transformation.
Finally, integrability condition gives $d\omega_2'=\omega\wedge\omega_2'=0$, which means that $\omega_2'=f\omega$
and $d(f\omega)=0$ for some rational function $f$. The latter condition, together with closedness of $\omega$, 
gives $df\wedge\omega=0$, i.e. $f$ is a first integral for $\mathcal F$.
\end{proof}

\begin{cor}\label{CorFclosedHclosed}
If $\mathcal F$ is defined by a closed $1$-form, then $\mathcal H$ is also defined by a closed $1$-form, 
or factors through a curve. If $\mathcal F$ admits a rational first integral, then $\mathcal H$ factors through a curve.
\end{cor}
\begin{proof}If $\mathcal F$ admits a rational first integral $f$, then after resolution of indeterminacy points of $f$ by blowing-ups $\tilde X\to X$,  Stein Factorization gives a fibration $\tilde f:\tilde X\to C$ over a curve with connected 
fibers coinciding generically with leaves of $\tilde{\mathcal F}$. Applying Lemma \ref{LemFClosedHClosed} to $\omega=df$,
we get that $\phi=\phi(\tilde f)$ and $\mathcal H$ actually factors through $\tilde f$.
If $\mathcal F$ is now defined by a closed $1$-form $\omega$, but does not admit a rational first integral,
then applying Lemma \ref{LemFClosedHClosed}, we get that $\phi=c\in\C$ is a constant, and 
$\mathcal H$ is defined by $\frac{dz}{1+cz^2} +\omega$ which is closed.
\end{proof}

\begin{remark}\label{RmFclosedHclosedLocal}
Statements similar to Lemma \ref{L:finalpiece} and Corollary \ref{CorFclosedHclosed} hold with the very same proofs in the local setting, 
replacing rational functions and $1$-forms by their meromorphic analogues. 
\end{remark}
\begin{proof}[Proof of Lemma \ref{L:trivializa}]
At a generic point $p\in\pi^{-1}(H)$, the foliation $\mathcal H$ is smooth and vertical: 
let $p_0=\pi(p)$ be the projection and $\sigma:(X,p_0)\to P$ be a germ of section transverse to $\mathcal H$. 
The induced transversely projective foliation is regular and therefore admits a holomorphic first integral.
By the local version of Corollary \ref{CorFclosedHclosed}, $\mathcal H$ locally factors through a curve.
\end{proof}

\subsection{Proof of Theorem \ref{THM:D}}
Under previous notations, we now apply to $(P,\mathcal H)$ our results on projective connections (Corollary \ref{COR:B} and Theorem \ref{THM:C}, 
or equivalently Theorem \ref{THM:E}). There are three cases.

{\bf First case.} There is a generically finite morphism $f:Y\to X$ such that $f^*\mathcal H$ is defined
by a closed $1$-form $\Omega$. The pull-back $f^*\mathcal F$ on $Y$ is still defined by restricting 
the Riccati foliation to the pull-back section $\tilde \sigma:Y\to f^*P$; it is therefore also defined by restricting
the closed $1$-form $\Omega$ to the section.

{\bf Second case.} There is a map $f:X\dashrightarrow C$ to a curve
and a Riccati $\mathcal H_0$ on $P_0=C\times\P^1$ such that $(P,\mathcal H)$ is equivalent to
$f^*(P_0,\mathcal H_0)$ by birational bundle transformation. We thus deduce a map
$\Phi:P\dashrightarrow P_0$ such that $\mathcal H=\Phi^*\mathcal H_0$. Consider the composition 
$\Phi\circ\sigma:X\dashrightarrow P_0$. Then either it is dominant and $\mathcal F$ is the corresponding pull-back  
of $\mathcal H_0$ (we are in case (2) of the statement), or the image is a curve and fibers force the leaves of $\mathcal F$ to be algebraic (we are in case (1)).

{\bf Third case.} There exists a rational map $f: X \dashrightarrow \mathfrak H$ to a polydisk Shimura modular orbifold
such that $(P,\mathcal H)$ is equivalent to the pull-back of one the tautological Riccati foliations  $(P_\rho:=\mathfrak H\times_\rho\P^1,\mathcal H_\rho)$ by birational bundle transformation. Again, we deduce a map
$\Phi:P\dashrightarrow P_\rho$ such that $\mathcal H=\Phi^*\mathcal H_\rho$, and considering
$\Phi\circ\sigma:X\to P_\rho$, we can conclude as before: we are in the case (1) or (3) of the statement 
depending whether the image is a curve or higher dimensional (not necessarily dominant).\qed

\subsection{Theorem \ref{THM:D} implies Theorem \ref{THM:E}}

Although we have followed the other direction, it is interesting to notice that our results on
projective (or $\sl$) connections and transversely projective foliations are actually equivalent.
Indeed, given say a Riccati foliation $ (\pi: P \to S,\mathcal H)$, with $P$ birationally trivial,
we can take a general rational section $\sigma:S\to P$ and consider the induced transversely projective foliation
$\mathcal F:=\sigma^*\mathcal H$.
Applying Theorem \ref{THM:D} to $\mathcal F$ gives the following possibilities.

{\bf First case.} The foliation $\mathcal F$ has a rational first integral. By Corollary \ref{CorFclosedHclosed},
we deduce that $\mathcal H$ factors through a curve.

{\bf Second case.} There is a generically finite morphism $f:Y\to X$ such that $f^*\mathcal F$ is defined
by a closed $1$-form $\omega$, but $\mathcal F$ does not admit a rational first integral.
By Corollary \ref{CorFclosedHclosed}, we deduce that $f^*\mathcal H$ is also defined by a closed $1$-form.

{\bf Third case.} Maybe after blowing-up $S$, there is a morphism $f:S\to P_0:=C\times \P^1$ 
and a Riccati foliation $\mathcal H_0$ on $P_0$ such that $\mathcal F=f^*\mathcal H_0$. 
Then, considering now the pull-back $\phi^*\mathcal H_0$ the fiber product:
$$\xymatrix{
	P \ar@{->}[r]^{\phi} \ar[d]_{\pi} & P_0 \ar[d]^{\pi_0} \\
	X \ar@{->}[r]^f  & C
}$$
we get another transversely projective structure for $\mathcal F$.
If $\mathcal F$ is not defined by a closed $1$-form up to finite cover, then its projective 
structure is unique, and $\mathcal H$ is birational to $\mathcal H_0$, thus pull-back from a curve.

{\bf Fourth case.} There exists a polydisk Shimura modular orbifold $\mathfrak H$ and a rational map $f: X \dashrightarrow P_\rho:=\mathfrak H\times_\rho\P^1$ towards one of its tautological Riccati foliations such that $\mathcal F= f^* \mathcal H_\rho$. Like in the previous case, we can prove that $\mathcal F$ is defined by a closed $1$-form after a finite cover
(and go back to the first two cases), or $\mathcal H$ is birationally equivalent to the pull-back of $(P_\rho,\mathcal H_\rho)$
by $f$.\qed

\section{Examples}\label{SecExamples}

We will now  present  some examples which show that our results are sharp.

\begin{example}
Let $Y$ be a projective manifold and consider a representation  $\rho : \pi_1(Y) \to (\C,+)$.   It determines a cohomology class $[\rho] \in H^1(Y,\C)$. If its image under the natural morphism $H^1(Y,\C) \to H^1(Y,\mathcal O_Y)$ is non zero then it determines a non trivial extension $0\to \mathcal O_Y \to E\to  \mathcal O_Y \to 0$, endowed with
a flat connection with monodromy given by 
\[
   \begin{pmatrix}1&\rho\\ 0&1\end{pmatrix}.
\] 
The projectivization $X=\mathbb P(E)$ is a $\mathbb P^1$-bundle over $Y$ with a Riccati foliation defined by a closed rational $1$-form $\omega$ with polar divisor equal to $2\Delta$, where $\Delta$ is image of the unique  section $Y \to \mathbb P(E)$.   If $P$ is the trivial $\mathbb P^1$-bundle $P$ over $X$ then we have a family of Riccati foliations $\mathcal H_{\lambda}$, $\lambda \in \C$, on it defined by
\[
dz + \omega ( 1 + \lambda z^2).
\]
The Riccati foliation $\mathcal H_{\lambda}$ is irregular for $\lambda \neq 0$,  does not factor through a curve, and its irregular divisor is not a
fiber of a fibration. If we take $Y$ equal to an elliptic curve, then $X - |\Delta|$ is nothing but Serre's example
of Stein quasi-projective surface which is not affine.
\end{example}

\begin{example}
Let $n\ge 2$ and let $X$ be the quotient of $\mathbb H^n$ by cocompact torsion-free irreducible lattice $\Gamma \subset \PSL(\mathbb R)^n$.
The natural projections $\mathbb H^n \to \mathbb H$ define $n$ codimension one smooth foliations on $X$ which
are transversely projective (indeed transversely hyperbolic). Contrary to what have been stated by the second author and Mendes in \cite[Theorem 1]{LG}, countably many leaves of these foliations may have nontrivial topology (with  fundamental groups  isomorphic to  isotropy groups of the action of $\Gamma$ on the corresponding one dimensional factor of $\mathbb H^n$),
 but the very general leaf is biholomorphic to $\mathbb H^{n-1}$. The maximal principle tells us that the general leaf cannot contain
 positive dimensional subvarieties, and consequently the foliations are not pull-backs from lower dimensional manifolds.
\end{example}

Again the assumptions on $\Gamma$ can be considerably weakened. All we have to ask is that $\Gamma$ is an irreducible lattice of $\PSL(\mathbb R)^n$ for some $n\ge 2$.
Notice that the rigidity theorem of Margulis (resp. the classification of representations by Corlette and Simpson) implies that all these lattices are commensurable  to (resp. conjugated to a subgroup of) arithmetic lattices of the form $U(P,\Phi)/ {\pm \Id}$ for some totally imaginary quadratic extension $L$ of a totally real number field $F$, some rank two projective $\mathcal O_L$-module $P$ and some skew Hermitian form $\Phi$. Besides the $n$ representations
coming from the $n$  projections $\pi_1^{orb}(X) \simeq \Gamma \subset \PSL(\mathbb R)^n \to \PSL(\mathbb R)$, we also have $[L:\mathbb Q] - 2n$ representations
of $\pi_1^{orb}(X)$ with values in $\PSL(\C)$ which do not factor through lower dimensional projective manifolds. The associated $\mathbb P^1$-bundles
are birationally trivial (since the underlying representation is a Galois conjugate of the representations coming from the transversely projective
foliations on $X$ defined by the submersions $\mathbb H^n \to \mathbb H$), and by taking a rational section we can produce further examples
of transversely projective foliations on $X$ which do not factor. Although the underlying representations are Galois conjugate to the representations
in $\PSL(\mathbb R)$, the topology of the Riccati foliations over $X$ associated to embeddings $\sigma : L \to \C$ for which
$\sqrt{-1} \Phi$ is definite is quite different. In the former case the Riccati foliation leaves invariant two open subsets, corresponding to the complement of $\mathbb P^1(\mathbb R)\simeq S^1$ in $\mathbb P^1$, while in the latter case the Riccati foliation is quasi-minimal: all the leaves not contained in $\pi^{-1}((\mathcal H)_{\infty})$
are dense in the corresponding $\mathbb P^1$-bundle.

Explicit examples of foliations on $\mathbb P^2$  defined by the submersions $\mathbb H^2 \to \mathbb H$ with $\Gamma \subset \PSL(2,\mathbb R)^2$ and $\Gamma$ isomorphic
to $\PSL(\mathcal O_K)$ (and certain subgroups) have been determined by the second author and Mendes in \cite{LG} ($K= \mathbb Q(\sqrt{5})$) using the work of
Hirzebruch on the description of these surfaces, and by Cousin
in \cite{Gael}, see also \cite{GaelTese}, ($K=\mathbb Q(\sqrt{3})$) using an  algebraic solution of  Painlev\'{e} VI equation.

\begin{example}
The degree of the generically finite morphism in Assertion (1) of Theorem \ref{THM:D} cannot be bounded, even if we restrict to transversely projective foliations
on a rational surface.  Let   $C_d = \{ x^d + y^d + z^d =0 \}  \subset \mathbb P^2$ be the
Fermat curve of degree $d \ge 3$. On $S_d = C_d \times C_d$ consider the action of $\mathbb Z/ d \mathbb Z$ given by  $\varphi(x,y) = ( \xi_d x ,  \xi_d y)$
where $\xi_d$ is a primitive $d$-th root of the unity. Let $ \omega \in H^0 (S_d, \Omega^1_{S_d})$ be a general holomorphic $1$-form satisfying $\varphi^* \omega = \xi_d \omega$.
The induced foliation is invariant by the action of $\mathbb Z/ d \mathbb Z$, but the $1$-form $\omega$ is not. The quotient of $S_d$ by $\mathbb Z/d\mathbb Z$
is a rational surface $R$  and the foliation induced by $\omega$ on $R$ is transversely projective (indeed transversely affine). The monodromy group is an
extension of the group of $d$-th of unities by an infinite  subgroup of $(\C, +)$. Explicit equations for birational models of these foliations on $\mathbb P^2$
can be found in \cite[Example 3.1]{jvpSad}.
\end{example}

\appendix

\section{Reduction of singularities of Riccati foliations}\label{SecProofMinimalModel}\label{S:normalform}

This appendix is devoted to the proof of Theorem \ref{THM:ReductionRiccati}.

Let $ (\pi: P \to S,\mathcal H)$ be a Riccati foliation over a projective surface $S$.
One can first blow-up $S$ until we get a simple normal crossing polar divisor $(\mathcal H)_\infty$,
and then apply elementary transformations over irreducible components of $(\mathcal H)_\infty$
until we minimize order of poles  on each component. This is explained in \cite{RH}. Then all components
of  $(\mathcal H)_\infty$  satisfy assumptions of Lemma \ref{L:trivializa} and, outside of special $\pi$-fibers,
the Riccati foliation locally factors into one dimensional models of Proposition \ref{prop:BrunellaTurritin}
or {\bf Irreg${}^{\text{ram}}$}. In particular, we can define the irregularity divisor as $I=\sum_i k_i D_i$
where $D_i$ run over irreducible components of $(\mathcal H)_\infty$, and $k_i\in \frac{1}{2}\Z_{\ge0}$
is the irregularity index.

To get rid of special points, we use the fact that $\pi: P \to S$ is birationally trivial and choose a rational section
$\sigma : S \dashrightarrow P$ which is not invariant by $\mathcal H$. 
The foliation $\mathcal F= \sigma^* \mathcal H$ is transversely projective; indeed, after birational
trivialization $z:P\to\mathbb P^1$ of the bundle, we may assume $\sigma$ given by $z=0$:
the foliation $\mathcal H$ defined by $dz+\omega_0+z\omega_1 +z^2\omega_2=0$
and $(\omega_0,\omega_1,\omega_2)$ is a projective triple for $\mathcal F$.
Now, apply Seidenberg's Theorem:
maybe blowing-up $S$, we can now assume that $\mathcal F$ has only reduced singular points.
Following the classification of Berthier and the third author \cite{BerthierTouzet,FredTransvProj},
reduced singular points of transversely projective foliations fall into one of the following types:
\begin{enumerate}
	\item\label{Case1stIntegral} $\mathcal F$ admits the holomorphic first integral:
	\begin{itemize}
		\item {\bf First integral:} $\omega=d(x^py^q)$ with $p,q\in\mathbb Z_{>0}$;
	\end{itemize}
	\item\label{CaseClosed1Form} $\mathcal F$ is defined by a closed $1$-form $\omega$ (but without first integral):
	\begin{itemize}
		\item {\bf Linear:} $\omega=\frac{dx}{x}+\lambda\frac{dy}{y}$ with $\lambda\in\C-\mathbb Q$;
		\item {\bf Saddle-node:} $\omega=\frac{dy}{y}+\frac{dx}{x^{k+1}}+\lambda\frac{dx}{x}$ with $\lambda\in\C$;
		\item {\bf Resonant saddle:} $\omega=\frac{dy}{y}+\frac{df}{f^{k+1}}+\lambda\frac{df}{f}$ with $f=x^py^q$ 
		and $\lambda\in\C$;
	\end{itemize}
	\item\label{CasePullBackRiccati} $\mathcal F$ is the pull-back of a singular point of a Riccati foliation by a ramified cover:
	\begin{itemize}
		\item {\bf Riccati saddle-node:} $\omega=dy+ \left(b(x)y^2+(\frac{1}{x^{k+1}}+\frac{\lambda}{x})y+c(x)\right)dx$;
		\item {\bf Bernoulli saddle-node:} $\omega=dz+ \left(b(x)z^2+(\frac{1}{x^{k+1}}+\frac{\lambda}{x})z\right)dx$ with $z=y^{\nu}$;
		\item {\bf Resonant saddle:} $\omega=dz+ \left(b(f)z^2+(\frac{1}{f^{k+1}}+\frac{\lambda}{f})z+c(f)\right)df$ with $f=x^py^q$ and $z=y^{\nu}$;
	\end{itemize}
\end{enumerate}

In case (\ref{Case1stIntegral}), the local version of Corollary \ref{CorFclosedHclosed} (see Remark \ref{RmFclosedHclosedLocal})
tells us that $\mathcal H$ factors through $f$, and we can reduce it to the models given by Proposition \ref{prop:BrunellaTurritin}.

In case (\ref{CaseClosed1Form}), where $\mathcal F$ is defined by a closed $1$-form $\omega$, we can
similarly reduce $\mathcal H$ to either $\Omega=dz+\omega$, or $\frac{dz}{z}+\lambda\omega$ with $\lambda\in\C^*$.

Finally, in case (\ref{CasePullBackRiccati}), we again see that $\mathcal H$ factors through $f=x$ or $f=x^py^q$,
reducing to the model 
$\Omega=dz+ \left(b(f)z^2+(\frac{1}{f^{k+1}}+\frac{\lambda}{f})z+c(f)\right)df$ and $\mathcal F$ is defined
by $z=y$ or $z=y^\nu$.

We summarize all the possibilities studied above in the next result. 

\begin{thm}\label{THM:ReductionRiccatiRamified}
	Let $ (\pi: P \to S,\mathcal H)$ be a Riccati foliation over a projective surface $S$,
	with $P$ a birationally trivial bundle.
	There exists a birational morphism $\phi:\tilde S\to S$ (with $\tilde S$ smooth) 
	and a birational bundle modification of the pull-back bundle $\phi^*P$ such that
	the pull-back Riccati foliation has normal crossing divisor $D$, and is locally defined
	over any point $p\in \vert D\vert$ by a Riccati $1$-form having one of the following types:
	\begin{enumerate}
		\item {\bf Log${}^{\mathbb G_m}$}: $\frac{\Omega}{z}=\frac{dz}{z}+\lambda_x\frac{dx}{x}$, or $\frac{\Omega}{z}=\frac{dz}{z}+\lambda_x\frac{dx}{x}+\lambda_y\frac{dy}{y}$;
		\item {\bf Log${}^{\mathbb G_a}$}: $\Omega=dz+\frac{dx}{x}$, or 
		$\Omega=dz+\frac{dx}{x}+\lambda\frac{dy}{y}$;
		\item {\bf Irreg${}^{0}$}: $\frac{\Omega}{z}=\frac{dz}{z}+\frac{df}{f^{k+1}}+\lambda_x\frac{dx}{x}+\lambda_y\frac{dy}{y}$,
		where $f=x$ or $x^py^q$;
		\item {\bf Irreg}:  $\Omega=dz+\left(\frac{df}{f^{k+1}}+\lambda\frac{df}{f}\right)z+\left(b(f)z^2+c(f)\right) df$, where $f=x$ or $x^py^q$;
		\item {\bf Irreg${}^{\text{ram}}$}: $\Omega=dz+(z^2+x^{\epsilon_1}y^{\epsilon_2}\phi(f))\frac{df}{x^{\tilde p}y^{\tilde q}f^k}-\left(\tilde p\frac{dx}{x}+\tilde q\frac{dy}{y}+k\frac{df}{f}\right)$ with $f$ and $\phi$ holomorphic, $\phi(0)\neq0$, $\epsilon_1,\epsilon_2=0,1$ and $(p,q)=(2\tilde p-\epsilon_1,2\tilde q-\epsilon_2)$.
	\end{enumerate}
\end{thm}

We note that the meromorphic gauge reduction process of Proposition \ref{prop:BrunellaTurritin}
which is {\it a priori} local can be done globally along irreducible components of the polar divisor.
Indeed, it can be checked on the proof of Proposition \ref{prop:BrunellaTurritin} that when the polar divisor 
is not minimal along an irreducible component $D_i$, then the non vertical part of the singular set 
of $\mathcal H$ is a (smooth) section; after an elementary transformation along it (blowing-up the section 
and then contracting the strict transform of $\pi^{-1}(D_i)$), the polar order decreases.
After a finite number of steps, the polar order is minimal, and all local models reduce to those of 
Proposition \ref{prop:BrunellaTurritin} by local biholomorphic bundle trivialization.
For details, see the proof of the main result of \cite{RH}.

The last step towards Theorem \ref{THM:ReductionRiccati} consists in passing to use a ramified covering
in order to kill ramifications. Going back to the irregular divisor $I=\sum_{i}k_iD_i$, we can choose
positive integers $m_i$ such that $m_ik_i=k$ for some fixed positive integer $k$ (a common multiple of all $k_i$'s).
We now consider a Kawamata covering (\cite[Proposition 4.1.12]{Lazarsfeld}): there exists a  ramified cover
$f:\tilde S\to S$ with $S$ smooth such that $f^*D_i=m_i\tilde D_i$ for some smooth reduced divisors $\tilde D_i$ on $\tilde S$,
and $\sum_iD_i$ has simple normal crossings. One easily check from models of Theorem \ref{THM:ReductionRiccatiRamified} that they are stable under ramified covers in variables $x$ or $y$,
and that ramifying at order $m_i$ along $D_i$ multiply the irregularity index by $m_i$.
Finally, after covering, the irregularity index is $k$ all along the irregular divisor. In particular, 
there are no more ramified components, and we have no more to consider the model {\bf Irreg${}^{\text{ram}}$}
of Theorem \ref{THM:ReductionRiccatiRamified}.\qed

\bibliographystyle{amsplain}

\end{document}